\documentclass[10pt]{article} 
\usepackage{amsthm}
\theoremstyle{definition}
\newtheorem{defi}{Definition}[section]
\newtheorem{remark}[defi]{Remark}

\newtheorem{example}[defi]{Example}

\theoremstyle{plain}
\newtheorem{theorem}[defi]{Theorem}
\newtheorem{cor}[defi]{Corollary}
\newtheorem{lemma}[defi]{Lemma}
\newtheorem{prop}[defi]{Proposition}

\newtheorem{asum}[defi]{Assumption}
\usepackage{cite}
 \usepackage{hyperref}
\usepackage{amsmath, amsfonts}
\usepackage[english]{babel}
\usepackage{enumerate}
\usepackage [T1]{fontenc}
\usepackage{endnotes}
\usepackage{xcolor}
\usepackage{abstract}
 \usepackage{cleveref}

\usepackage{enumitem,blindtext}
\usepackage[paper=a4paper,left=30mm,right=30mm,top=30mm,bottom=30mm]{geometry}
\allowdisplaybreaks

\numberwithin{equation}{section}

\newcommand{\blue}[1]{\textcolor{black}{#1}}

\newcommand{\green}[1]{\textcolor{black}{#1}}
\newcommand{\greenNew}[1]{\textcolor{black}{#1}}

\newcommand*\samethanks[1][\value{footnote}]{\footnotemark[#1]}

\title{Multi-dimensional fractional Brownian motion in the $G$-setting}
\author{Francesca Biagini\thanks{Workgroup Financial and Insurance Mathematics, Department of Mathematics, Ludwig-Maximilians Universit{\"a}t, Theresienstrasse 39, 80333 Munich, Germany. Emails: biagini@math.lmu.de, oberpriller@math.lmu.de.} \and Andrea Mazzon\thanks{University of Verona, Via Cantarane 24, 37129 Verona, Italy. Email: andrea.mazzon@univr.it.} \and Katharina Oberpriller\samethanks[1]}

\begin{document}
\maketitle
\begin{abstract}
In this paper we introduce a definition of a {multi}-dimensional fractional Brownian motion of Hurst index $H \in (0,1)$ under volatility uncertainty {(in short $G$-fBm)}. We study the properties of such a process and provide first results about stochastic calculus with respect to a \blue{multi-dimensional} $G$-fBm for a Hurst index \blue{$H\in (\frac{1}{2},1)$}.
\end{abstract}
\green{\textbf{Keywords:}
Volatility uncertainty, fractional Brownian motion, pathwise stochastic integral}

\section{Introduction}

The modeling of randomness through a possibly multi-dimensional Brownian motion has been encountering in the last decades significant challenges from two distinct perspectives when coming to financial applications. On \blue{the} one side, various irregularities which are observed in the trajectories of time series {of asset prices} are not adequately described by a classical Brownian motion. On the other side, {modelization needs to take into account the risk connected to the choice of one particular probability measure.}

 The main goal of our work is to {address both issues by }defin{ing} a multi-dimensional fractional Brownian motion in the so called $G$-setting {introduced} by Peng in  \cite{peng_nonlinearExpectation_book}, which addresses volatility uncertainty and ambiguity aversion of traders in a framework with multiple, possibly non-dominated, probability measures.

{F}ractional Brownian motion (fBm), introduced in 1940 by Kolmogorov for modeling turbulence in liquids in \cite{kolmogorov1940wienersche}, has been {used} in \cite{gatheral2018volatility} to account for the volatility roughness, or irregularity, empirically observed in high-frequency data. The fractional Brownian motion also describes further behaviours observed in time series, like long-range dependence, non-stationarity and self-similarity, which cannot be explained within a pure Brownian motion framework. Several contributions in the last years have then proposed to ``roughen'' classic Markovian stochastic volatility models such as Hull-White, Heston and SABR by replacing the Gaussian driver of the volatility process by a fractional one, see for example \cite{ElEuchRosenbaum_2018}, \cite{ElEuchFukasawaRosenbaum_2018} and \cite{ElEuchRosenbaum_2019}.

The fractional Brownian motion is {a class of processes indicized} by the so called Hurst index $H \in (0,1)$, describing its roughness. Apart from the case $H = \frac 1 2$, when it boils down to the classical Brownian motion, the fractional Brownian motion is not a semimartingale, which makes it particularly challenging to mathematically deal with it. For an overview about fractional Brownian motion and its properties and applications we refer to \cite{nourdin_fractional_brownian_motion} {and \cite{biagini_hu_oksendal_zhang_2008}.}

{M}odel uncertainty, \blue{i.e.} the fact that some features of the stochastic processes under investigation can be determined only up to a certain degree of precision, has also gained more importance over the last years. A stream of literature covers the case of uncertainty about the drift, see e.g. \cite{chen2002ambiguity}, \cite{cheng2013optimal}, \cite{balter2021time}, \cite{christensen2013optimal}. On the other hand, further approaches focus on volatility uncertainty, see among others \cite{avellaneda1995pricing}, \cite{lyons1995uncertain}, \cite{nutz2012superhedging}, \cite{muhle2018risk}. Under this framework, the set of probability measures is in general non-dominated, which makes things more difficult. In particular, the $G$-Brownian motion calculus has been introduced by Shige Peng, see e.g. \cite{peng_nonlinearExpectation_book}, to account for volatility uncertainty. In this framework, the expectation operator is not linear but only sublinear, and fundamental distinctions from the classical setting arise: for example,  the quadratic variation of a $G$-Brownian motion is not deterministic as under a single prior, but rather a stochastic process. Due to these profound differences, it is not straightforward to extend existing concepts and definitions {under} volatility uncertainty. However, several standard results of classical stochastic analysis have been carried to the $G$-setting: examples include a $G$-It\^{o} formula, a martingale representation theorem, and the existence and uniqueness of the solution to a $G$-Stochastic Differential Equation (\blue{$G$}-SDE). On the level of stochastic processes, a major achievement was the definition of a $G$-L\'{e}vy process with independent and stationary increments and jumps in the $G$-setting, see \cite{Hu_Peng_Levy}.

It is then a natural question how to introduce a fractional Brownian motion under volatility uncertainty. A first attempt in this direction is provided in \cite{fBrownianMotionGSetting1} and \cite{chen_2013} for defining a one-dimensional fractional Brownian motion {under volatility uncertainty}. The approach in \cite{chen_2013} directly transfers the definition of a fractional Brownian motion as a centered Gaussian process with a given covariance function into the $G$-setting.  However, this gives rise to several problems, as also highlighted in \cite{fBrownianMotionGSetting1}. Notably, the definition of a $G$-Gaussian process in \cite{chen_2013} diverges from the one introduced by Peng in \cite{peng_Gaussian_2011}. % 
Furthermore, the proof of the existence of such a fractional $G$-Brownian motion relies on a two-sided $G$-Brownian motion, whose existence is not clear. For a detailed discussion on these problems, we refer to Section {\ref{sec:Multi-DimensionalGfBm}}. 

These issues have been overcome by the definition of a one-dimensional fractional $G$-Brownian motion in \cite{fBrownianMotionGSetting1}. Here, a fractional $G$-Brownian motion is introduced as reported in Definition \ref{def:fBrownianMotionSetting1}. The authors prove the existence of a process satisfying this definition and establish a representation theorem by expressing it as an integral of a stochastic kernel and a volatility function with respect to a standard Brownian motion. Additionally, it is shown that the fractional $G$-Brownian motion exhibits properties such as self-similarity and long-range dependence, among others. 

{In this paper we study the problem of providing a well-posed} definition of a fractional $G$-Brownian motion in higher dimensions. Extending  Definition \ref{def:fBrownianMotionSetting1} to the multi-dimensional case is a priori not possible. First, the components of a multi-dimensional $G$-Brownian motion are not independent, differently from those of a standard Brownian motion. This is {clear by} looking at the quadratic variation process of a $d$-dimensional $G$-Brownian motion, as the off-diagonal elements are stochastic processes different from zero. For this reason, one needs to account for the covariance structure inherent to the multi-dimensional fractional $G$-Brownian motion, originating from the covariance structure of a multi-dimensional $G$-Brownian motion. 

{Furthermore, we need to avoid} to rely on a  two-sided $G$-Brownian motion, which introduces several mathematical difficulties as explained before. {To this purpose}, in Definition \ref{def:FractionalGBrownianMotionOurDDimension} we introduce a multi-dimensional fractional $G$-Brownian motion {($G$-fBm)} as an integral with respect to a multi-dimensional $G$-Brownian motion, generalizing the Volterra representation of a standard fractional Brownian motion {of} Proposition 2.5 in \cite{nourdin_fractional_brownian_motion}. 
We show that for $d=1$ our definition of a fractional $G$-Brownian motion also adheres to the characterizations of Definition \ref{def:fBrownianMotionSetting1}. We study several properties of the \blue{multi-dimensional} fractional $G$-Brownian motion such as \blue{its} covariance function, self-similarity and the distribution of the random variable{s} $B_t^H(i)$ for $t \geq 0$, \blue{$H \in (0,1)$} and \blue{$i=1,...,d$}. Furthermore, we prove componentwise properties of the increments of our process, such as stationarity and the behaviour of the product of the increments from two different components. Additionally, we analyze the long and short memory property of a multi-dimensional fractional $G$-Brownian motion. We also focus on the paths properties of the process, which are crucial for modelling financial time series data. Among others\blue{,} we show that a \blue{multi-dimensional} fractional $G$-Brownian motion of Hurst index \blue{$H\in (0,1)$} admits a modification with H\"older continuous paths of order $\alpha <H$ and is nowhere differentiable, quasi-surely. Most of these properties are generalizations of the classical ones under volatility uncertainty. However, the proofs are more involved, as we can no longer rely on the Gaussian property but need to come up with other techniques. 

For applications, it is crucial to {introduce stochastic calculus} with respect to a fractional $G$-Brownian motion. In the classical case of a fractional Brownian motion {there are }several different ways {of} defin{ing} an integral with respect to a fractional Brownian motion, as it is not possible to use the It\^o integral. All these different approaches can be classified in two methods. The first method is a pathwise construction of the integral, relying on the paths properties of the fractional Brownian motion, see e.g. \cite{zaehle} and \cite{nualart2002}. The second method is based on Wick products or Malliavin calculus, see e.g. \cite{Hu_Oksendal_2003} or \cite{Decreusefond_1999}. Here, we refer only to the first works in these areas and stress that the two approaches have been extensively developed further into different directions, see {\cite{biagini_hu_oksendal_zhang_2008} and \cite{mishura_fractional_BM}.} {Here,} we {follow the} method of \cite{nualart2002} to introduce a {pathwise} stochastic integral with respect to a fractional $G$-Brownian motion \blue{of} Hurst index \blue{$H\in(\frac{1}{2},1)$}. The definition allows us to derive an existence result for a pathwise differential equation driven by a fractional $G$-Brownian motion, which is the first result in the literature studying such equations. We prove a pathwise It\^{o}'s formula for \blue{a fractional $G$-Brownian motion of Hurst index $H\in(\frac{1}{2},1)$}.  \blue{We conclude the manuscript illustrating an application of this result by analyzing the existence of arbitrages in fractal financial markets under volatility uncertainty.}

{The paper is organized as follows. In Section \ref{sec:Gsetting} we provide an overview about the $G$-expectation framework. In Section \ref{sec:Multi-DimensionalGfBm} we introduce our definition of a multi-dimensional fractional $G$-Brownian motion and study its properties. The paper is concluded by deriving a pathwise stochastic integration for \blue{a} fractional $G$-Brownian motion {of Hurst index $H \in \greenNew{(0,\frac{1}{2})}$} in Section \ref{sec:StochasticCalculus}.}

\section{$G$-Setting}\label{sec:Gsetting}
We now recall the basic concepts for the $G$-setting based on \cite{peng_nonlinearExpectation_book}.
\begin{defi}\label{def:initial}
Let $\Omega$ be a given set and let $\mathcal{H}$ be a vector lattice of $\mathbb{R}$-valued functions defined on $\Omega$ such that $c \in\mathcal{H}$ for all constants $c$, and $\vert X \vert \in \mathcal{H},$ if $X \in \mathcal{H}$. $\mathcal{H}$ is considered as the space of random variables. A \emph{sublinear expectation} ${\mathbb{E}}$ on $\mathcal{H}$ is a functional ${\mathbb{E}}: \mathcal{H} \to \mathbb{R}$ satisfying the following properties: for all $X,Y \in \mathcal{H}$, we have
\begin{enumerate}
	\item Monotonicity: If $X \geq Y$ then $\mathbb{E}[X] \geq \mathbb{E}[Y]$\blue{;}
\item Constant preserving: $\mathbb{E}[c]=c$ $\blue{\forall c \in \mathbb{R}}$\blue{;}
\item Subadditivity: $\mathbb{E}[X+Y] \leq \mathbb{E}[X]+\mathbb{E}[Y]$\blue{;}
\item Positive homogenity: $\mathbb{E}[\lambda X]= \lambda \mathbb{E}[X]$, $\forall \lambda >0$. 
\end{enumerate}
The triple $(\Omega, \mathcal{H}, \mathbb{E})$ is called a \emph{sublinear expectation space}. If only 1. and 2. are satisfied, $\mathbb{E}$ is called a \emph{nonlinear expectation} and the triple $(\Omega, \mathcal{H}, \mathbb{E})$ is called a \emph{nonlinear expectation space}.
\end{defi}
{Fix $d \in \mathbb{N}$}. Let $\mathcal{H}$ be a space of random variables such that if $X_i \in \mathcal{H}, i=1,...,d$, then 
\begin{equation*}
	\varphi(X_1,...,X_d) \in \mathcal{H}, \quad \text{ for all } \varphi \in C_{l,lip}(\mathbb{R}^d),
\end{equation*}
where $C_{l,lip}(\mathbb{R}^d)$ denotes the linear space of {$\mathbb{R}$-valued} functions satisfying for each $x,y \in \mathbb{R}^d$,
\begin{equation}
	\vert \varphi(x)-\varphi(\blue{y}) \vert \leq C_{\varphi}(1+ \vert x \vert^m + \vert y \vert^m) \vert x-y \vert, 
\end{equation}
for some $C_{\varphi}>0, m \in \mathbb{N}$ depending on $\varphi$.
In this case, $X=(X_1,...,X_d)$ is called a $d$-dimensional \emph{random vector}, denoted by $X \in \mathcal{H}^d$. 

We now introduce \blue{a} $G$-Brownian motion by giving the following definitions.
\begin{defi} \label{defi:independenceRandomVariables}
In a sublinear expectation space $(\Omega,\mathcal{H},\mathbb{E})$, a random vector $Y \in \mathcal{H}^{d}$ is said to be \emph{independent} \blue{of} another random vector $X \in \mathcal{H}^m$ \blue{for $m\in \mathbb{N}$} under $\mathbb{E} [\cdot]$ if for each test function $\varphi \in C_{l,lip}(\mathbb{R}^{d+m})$ we have
$$\mathbb{E}[\varphi(X,Y)]=\mathbb{E}\left[\mathbb{E}[\varphi(x,Y)]\vert_{x=X}\right].$$ 
 \end{defi}
\begin{defi}
Let $X$ and $Y$ be two $d$-dimensional random vectors defined on sublinear expectation spaces $(\Omega_1,\mathcal{H}_1, \mathbb{E}_1)$ and $(\Omega_2,\mathcal{H}_2, \mathbb{E}_2)$, respectively. They are called \emph{identical{ly} distributed} if $$\mathbb{E}_1[\varphi(X)]=\mathbb{E}_2[\varphi(Y)] \quad \text{for {all} }  \varphi \in C_{{l},lip}(\mathbb{R}^{d}).$$ In this case, we write $X \sim Y$.
\end{defi}

\blue{\begin{defi}
Let $X$ and $Y$ be two $d$-dimensional random vectors on a sublinear expectation space $(\Omega, \mathcal{H}, \mathbb{E})$. Then $X$ is called a \emph{independent copy} of $Y$ if $X$ is independent of $Y$ and $X \sim Y$.
\end{defi}}

\begin{defi} \label{defi:G-MaximalDistribution}
A random variable $\eta$ on a sublinear expectation space $(\Omega, \mathcal{H}, \mathbb{E})$ is called \emph{maximally distributed} if there exists a bounded and closed subset $\Gamma \in \mathbb{R}^d$ such that 
$$
\mathbb{E}[\varphi(\eta)]=\max_{y \in \Gamma} \varphi(y), \quad \varphi \in C_{l\blue{,}lip}(\mathbb{R}^d).
$$
\end{defi}

\begin{defi} \label{defi:G-NormalDistribution}
A {$d$-dimensional} random vector $X$ \blue{on} a sublinear expectation space $(\Omega, \mathcal{H}, \mathbb{E})$ is called \emph{$G$-normally distributed} if for each $a,b \geq 0$ it holds
\begin{equation*}
	a X + b \overline{X} \sim \sqrt{a^2 + b^2}X, 
\end{equation*}	
where $\overline{X}$ is {any} independent copy of $X$. The letter $G$ denotes the function $G: \mathbb{S}(d) \to \mathbb{R}$ defined by
\begin{equation}
	G(A):=  \frac{1}{2} \mathbb{E}\left[ \langle A X, X \rangle \right], \label{G-equationOneDimension_1}
\end{equation}
where $\mathbb{S}(d)$ is the collection of all symmetric $(d \times d)$-matrices \blue{and $\langle \cdot, \cdot \rangle$ denotes the standard dot product on $\mathbb{R}^d$}.
\end{defi}
Note that $G$ is a continuous, sublinear function which is monotone in $A \in \mathbb{S}(d)${, i.e. for each $p, \bar{p} \in \mathbb{R}^d$ and $A, \bar{A} \in \mathbb{S}(d)$ it holds
$$
\begin{cases}
G(p+\bar{p}, A+\bar{A})  \leq G(p, A)+G(\bar{p}, \bar{A}), \\
G(\lambda p, \lambda A)  =\lambda G(p, A), \quad \forall \lambda \geq 0, \\
G(p, A)  \leq G(p, \bar{A}), \quad \text { if } A \leq \bar{A}.
\end{cases}
$$} 
Moreover, there exists a bounded, closed and convex subset $\Theta \subseteq \mathbb{S}(d)$ such that
\begin{equation}
	G(A)=  \frac{1}{2} \sup_{\gamma \in \Theta} \textnormal{tr}[A \gamma], \label{G-equationOneDimension}
\end{equation}
{see equation (2.2.4) of \cite{peng_nonlinearExpectation_book}.}
{The existence of a vector $X$ which is $G$-normally distributed according to Definition \ref{defi:G-NormalDistribution} is given in Proposition 2.2.7 of \cite{peng_nonlinearExpectation_book}. Exploiting the representation \eqref{G-equationOneDimension}, we identify $\Theta$ as the uncertainty set for the covariance of $X$ and denote $
X \sim \mathcal{N}\left( 0, \Theta\right)$.
}
	
\begin{defi} \label{defi:G-BrownainMotion}
	Let $G: \mathbb{S}(d) \to \mathbb{R}$ be a given monotonic and sublinear function. A $d$-dimensional stochastic process $B=(B_t)_{t \geq 0}$ on a sublinear expectation space $(\Omega, \mathcal{H}, \mathbb{E})$ is called a \emph{$G$-Brownian motion} if it satisfies the following conditions:
	\begin{enumerate}
		\item $B_0=0$;
		\item $B_t \in \mathcal{H}^d$ for each $t \geq 0$;
		\item For each $t, s \geq 0$ the increment $B_{t+s}-B_t$ is independent of $(B_{t_1},...,B_{t_n})$ for each $n \in \mathbb{N}$ and $0 \leq t_1 <...<t_n \leq t$. Moreover, $(B_{t+s}-B_t)s^{-1/2}$ is $G$-normally distributed. 
	\end{enumerate}
\end{defi}

{\begin{defi} \label{defi:G-Gaussian}
An $\mathbb{R}^d$-valued stochastic process $X=(X_t)_{t \geq 0}$ defined on a sublinear expectation space {$(\Omega,\mathcal{H}, \mathbb{E})$} is called a \emph{$G$-Gaussian process} if for each $\underline{t}=(t_1,..,t_n) \in \mathcal{T}$, $X_{\underline{t}}=(X_{t_1},...,X_{t_n})$ is an $\mathbb{R}^{ d \times n}$-valued normally distributed random variable. Here, we denote 
\begin{equation} \label{eq:VectorTimes}
\mathcal{T}:= \left\lbrace \underline{t}=(t_1,...,t_m):  m \in \mathbb{N}, t_i \in [0,\infty), t_i {<} t_j, 0 \leq i {<} j \leq m, i \neq j\right\rbrace.
\end{equation}
\end{defi}
\begin{remark} \label{remark:GBrownianMotionNotGaussian}
	In \cite{peng_Gaussian_2011} it is mentioned that in contrast to the classical situation, a $G$-Brownian motion is not a \blue{$G$-}Gaussian process although each increment of a $G$-Brownian motion is \blue{$G$-}normally distributed.
\end{remark}
}

{We now construct a $d$-dimensional $G$-Brownian motion.}
We denote by $\Omega:=C_0(\mathbb{R}_+,\mathbb{R}^d)$ the space of all $\mathbb{R}^d$-valued continuous paths $\omega=(\omega_t)_{\blue{t \ge 0}}$ with $\omega_0=0$ and equip this space with the distance
\begin{equation} \label{eq:Distance}
\rho(\omega^1,\omega^2):=\sum_{i=1}^{\infty}2^{-i} \left[ \left( \max_{t \in [0,i]} \left \vert \omega_t^1 - \omega_t^2 \right \vert \right) \wedge 1 \right].
\end{equation}
For every fixed $T \in [0,\infty)$, we define $\Omega_T:=\lbrace \omega_{\cdot \wedge T}: \omega \in \Omega \rbrace$.
We set $\mathcal{F}:=\mathcal{B}(\Omega)$ and $\mathcal{F}_T:=\mathcal{B}(\Omega_T)$ which denotes the corresponding Borel $\sigma$-algebras. Moreover, the canonical process $B=(B_t)_{t\geq 0}$ on $\Omega$ is given by $B_t(\omega):=\omega_t$ for $\omega \in \blue{\Omega}, t \geq 0$.  
\begin{defi}
For each fixed $T \in [0, \infty)$ we introduce $Lip(\Omega_T)$ as
\begin{equation*}
	Lip(\Omega_T):=\left \lbrace \varphi (B_{t_1 \wedge T},...,B_{t_n \wedge T}): n \in \mathbb{N}, t_1,...,t_n \in [0,\infty), \varphi \in C_{l,\blue{lip}}(\mathbb{R}^{d\times n}) \right \rbrace.
\end{equation*}
We also set
\begin{equation*}
	Lip(\Omega):=\bigcup_{n=1}^{\infty} Lip(\Omega_n).
\end{equation*}
 \end{defi}
Next we introduce the definition of $G$-expectation and conditional $G$-expectation. 
\begin{defi}
A \emph{$G$-expectation} $\hat{\mathbb{E}}$ is a sublinear expectation on $(\Omega,Lip(\Omega))$  defined as follows: For $X \in Lip(\Omega)$ of the form 
\begin{equation}\label{eq:Xpeng}
	X=\varphi\left(B_{t_1}-B_{t_0},...,B_{t_n}-B_{t_{n-1}}\right), \quad 0 = t_0 < t_1<...<t_n <\infty, \quad \blue{n \in \mathbb{N}}, 
\end{equation}
for some $\varphi \in C_{\blue{l,lip}}(\mathbb{R}^{d \times n})$ we set 
\begin{equation*}
	\hat{\mathbb{E}}[X]:={\tilde{\mathbb{E}}}\left[\varphi\left(\xi_1 \sqrt{t_1-t_0},...,\xi_n \sqrt{t_n - t_{n-1}
	}\right)\right],
\end{equation*}
where $\xi_1,...,\xi_n$ are $d$-dimensional random variables on a sublinear expectation space $(\tilde{\Omega},\tilde{\mathcal{H}}, {\tilde{\mathbb{E}}})$ such that for each $i=1,...,n$, $\xi_i$ is $G$-normally distributed and independent of $(\xi_1,...,\xi_{i-1})$.
\end{defi}
{Then the canonical process} $B=(B_t)_{\blue{t \ge 0}}$ is a $G$-Brownian motion on the sublinear expectation space $\blue{(\Omega, Lip(\Omega)}, \hat{\mathbb{E}})$. 
{\begin{remark}
	Note that the elements in $Lip(\Omega)$ take values in $\mathbb{R}$.
\end{remark}}
\begin{defi}
The space $L_G^p(\Omega)$ is the completion of $Lip(\Omega)$ under the norm $$\| \xi \|_p=\left(\hat{\mathbb{E}}[\vert \xi \vert^p]\right)^{{1}/{p}}, \quad p \geq 1.$$
Similarly, we can also define $L_G^p(\Omega_T)$ for $p \geq 1$ and $T \in [0,\infty)$.	
\end{defi}

\begin{defi}
	Let $X \in Lip(\Omega)$ have the representation \blue{\eqref{eq:Xpeng}}.
Then \blue{its} \emph{$G$-conditional expectation} under $\Omega_{t_j}$ is defined as
\begin{equation*}
	\hat{\mathbb{E}}\left[X \vert \Omega_{t_j}\right]:= \psi \left(B_{t_1}-B_{t_0},B_{t_2}-B_{t_1},...,B_{t_j}-B_{t_j-1}\right), \quad j=1,\dots,n-1,
\end{equation*}
where
\begin{equation*}
	\psi(x):=\hat{\mathbb{E}}\left[\varphi \left(x,B_{t_j+1}-B_{t_j},...,B_{t_n}-B_{t_n-1}\right)\right].
\end{equation*}
As for any $ T< \infty$, $Lip(\Omega_T) \subseteq Lip(\Omega)$ we can also consider $\hat{\mathbb{E}}[\cdot], \hat{\mathbb{E}}[\cdot \vert \Omega_t]$ as functions on $Lip(\Omega_T)$ for $t \leq T$.

\end{defi}
The $G$-expectation and the $G$-conditional expectation can be extended to sublinear operators $$\hat{\mathbb{E}}[\cdot]:L_G^p(\Omega) \to \mathbb{R}$$ and $$\hat{\mathbb{E}}[\cdot \vert \Omega_t]: L_G^1(\Omega) \to L_G^1(\Omega_t),$$
see Section 3.2 in \cite{peng_nonlinearExpectation_book}. Furthermore, $\hat{\mathbb{E}}[\cdot], \hat{\mathbb{E}}[\cdot \vert \Omega_t]$ can also be considered on the space $L_G^1(\Omega_T)$ for any $T \in [0,\infty)$ \blue{and $t \leq T$}.\\
In the following definitions we introduce {some} spaces we will work with {in the sequel}.
\begin{defi} \label{DefinitionSimpleIntegrands}
\begin{enumerate}
\item For $p \geq 1$ we denote by $M^{p,0}_G(0,T)$ the space of \emph{simple integrands}. Specifically, we define an element $\eta \in M^{p,0}_G(0,T)$ by
\begin{equation} \label{eq:SimpleFunction}
	\eta_t(\omega)=\sum_{j=0}^{N-1} \xi_j(\omega) \textbf{1}_{[t_j, t_{j+1})}(t),
\end{equation}
\blue{for $N \in \mathbb{N}$, where $\xi_i \in L_G^p(\Omega_{t_i}), i=0,...,N-1$ and $0=t_0<t_1<...<t_N=T$}.
\item For $p \geq 1$ we let $M_G^p(0,T)$ be the completion of $M^{p,0}_G(0,T)$ under the norm $$\left(\hat{\mathbb{E}}\left[\int_0^T \vert \eta_t \vert^p dt\right]\right)^{1/p}.$$
\end{enumerate}
\end{defi}
Similar as in the classical It\^{o} case, the integral with respect to a one-dimensional $G$-Brownian motion $B$ is first defined for the simple integrands in \blue{$M_G^{2,0}(0,T)$} by the mapping 
\blue{\begin{equation*}
	I:M_G^{2,0}(0,T) \to L_G^2(\Omega_T), \quad I(\eta)\blue{:=}\greenNew{\sum_{j=0}^{N-1} \eta_j\left( B_{t_{j+1}}-B_{t_j}\right),}
\end{equation*}
\greenNew{for $\eta$ as in \eqref{eq:SimpleFunction}}.}
Then this mapping can be continuously extended to 
$$I:  M_G^{2}(0,T) \to L_G^2(\Omega_T),$$ see Lemma 3.3.4 in \cite{peng_nonlinearExpectation_book}. In order to define the integral with respect to a $d$-dimensional $G$-Brownian motion we proceed as follows. For a fixed $\textbf{a} \in \mathbb{R}^d$, we define the one-dimensional process $B^{\textbf{a}}:=(B^{\textbf{a}}_t)_{t \geq 0}$ \blue{by} 
\begin{equation} \label{eq:DefinitonScalarGBrownianMotion}
	B^{\textbf{a}}_t:=\langle \textbf{a}, B_t \rangle, \quad t \geq 0.
\end{equation}
Then $B^{\textbf{a}}$ is a one-dimensional $G_{\textbf{a}}$-Brownian motion with 
\begin{equation} \label{eq:FunctionGScalar}
	G_{\textbf{a}}(\alpha)= \frac{1}{2}\left( \sigma^2_{\textbf{a}\textbf{a}^T} \alpha^+ - \sigma^2_{-\textbf{a}\textbf{a}^T}\alpha^-  \right),
\end{equation}
where 
\begin{align} \label{eq:DefinitionBoundariesVolatilityScalar}
	\sigma_{\textbf{a}\textbf{a}^T}^2= 2 G\left( \textbf{a}\textbf{a}^T\right) \quad \text{ and } \quad \sigma_{-\textbf{a}\textbf{a}^T}^2= -2 G\left( -\textbf{a}\textbf{a}^T\right). 
\end{align}
Then, the \blue{$G$-}It\^{o} integral can be defined as
\begin{equation} \label{eq:DefinitionItoScalarProduct}
	I(\eta):= \int_0^T \eta_t dB_t^{\textbf{a}}, \quad \eta \in M_G^2(0,T).
\end{equation}

\begin{defi} \label{defi:QuadraticVariation}
	Let $B=(B_t)_{t \geq 0}$ be a one-dimensional $G$-Brownian motion. Let \blue{$\pi_t^N:=\lbrace t_0^N,...,t_N^N \rbrace$, $N\in \mathbb{N}\greenNew{,}$} be a sequence of partitions of $[0,t]$ \blue{and call} 
	\blue{
	\begin{align} \label{eq:MeshSize}
	\mu(\pi_t^N):=\max \lbrace  \vert t_{i+1}^N-t_{i}^N \vert : i=0,1,...,N-1 \rbrace	
	\end{align}}
the mesh size of $\pi_t^N$, \blue{$N\in \mathbb{N}$}.  The \emph{quadratic variation process of \blue{a} $G$-Brownian motion} $B$, denoted by $\langle B \rangle=(\langle B \rangle_t)_{t \geq 0}$, is given by
	$$
	\langle B \rangle_t := \lim_{\mu(\pi_t^N) \to 0} \sum_{j=0}^{N-1} \left(B_{t_{j+1}^N}-B_{t_{j}^N} \right)^2=B_t^2 - 2 \int_0^t B_s dB_s, \quad t \in [0,T].
	$$
\end{defi}
Note that the process $\langle B \rangle$ is not deterministic except for the classical case, where $B$ is a classical Brownian motion. In particular, the quadratic variation process contains the part of statistic uncertainty of \blue{a} $G$-Brownian motion. The integral with respect to the quadratic variation $\int_0^{\cdot} \eta_s d \langle B \rangle_s$ is introduced first for $\eta \in M_G^{1,0}(0,T)$ and then extended to $M_G^1(0,T)$, see Section 3.4 in \cite{peng_nonlinearExpectation_book}.

The quadratic variation process of a $d$-dimensional $G$-Brownian motion \blue{takes} values in $\mathbb{S}(d)$, and in particular it is defined as {the matrix $\langle B \rangle$ with components}
\begin{equation} \label{eq:QuadraticVariationD}
{[\langle B \rangle_t]_{i,j}}:=\langle B(i), B(j) \rangle_{t}, \quad t \ge 0, \quad {i,j=1,...,d,}
\end{equation}
where $B(i)=(B(i)_t)_{t \geq 0}$ denotes the $i$th component of \blue{a} $G$-Brownian motion. Note that $B(i)_t=B^{\textbf{e}_i}_t$ with $\textbf{e}_i=(\textbf{e}_i(j))_{j=1,...,d} \in \mathbb{R}^d$ with $\textbf{e}_i(j)=\delta_i(j)$, where $\delta_i$ denotes the Dirac delta for $i$.
In particular, {we can define for each $i,j=1,...,d$ and $t \geq 0$}
\begin{equation} \label{eq:CovariationMatrixGBrownianmotionD}
	\langle B(i),B(j) \rangle_t:=\frac{1}{4}\left( \langle B(i)+B(j) \rangle_t -\langle B(i)-B(j) \rangle_t \right),
\end{equation}
{see page 64 of \cite{peng_nonlinearExpectation_book}.}
 For any $t \geq 0$ it holds 
\begin{equation}
	\hat{\mathbb{E}} \left[ \langle A B_t, B_t \rangle \right]=2G(A) t \quad \text{for } A \in \mathbb{S}(d),
\end{equation}
see page {70} in \cite{peng_nonlinearExpectation_book}. Thus, 
\begin{equation}
	\hat{\mathbb{E}} \left[ B_t(i)B_t(j) \right]=\blue{2}G(\bar{A}) t =t  \sup_{\gamma \in \Theta} \gamma_{ij}, \label{eq:QuadraticVariationD_new}
\end{equation}
with $\bar{A}=(\bar{a}_{kl})_{k,l=1,...,d}$ such that $\bar{a}_{kl}=0$ for $(k,l) \blue{\notin} \lbrace(i,j),(j,i) \rbrace$ and $\bar{a}_{ij}=\bar{a}_{ji}=\blue{\frac{1}{2}}$. \blue{Moreover,} $\Theta$ is defined in \eqref{G-equationOneDimension}.

\begin{defi}
A one-dimensional process $(M_t)_{t \in [0,T]}$ is called a \emph{$G$-martingale} (respectively, \blue{\emph{$G$-supermartingale, $G$-submartingale}}) if for each $t \in [0,T]$, $M_t \in L_G^1(\Omega_t)$ and for each $0 \leq s \leq t \leq T$, we have
\begin{equation*}
	\hat{\mathbb{E}}[M_t\vert \Omega_s]=M_s, \quad \text{(respectively, } \leq M_s, \quad \geq M_s \text{).}
\end{equation*}
If in addition it holds also 
\begin{equation*}
	\hat{\mathbb{E}}[-M_t \vert \Omega_s]=-M_s \quad\text{ for }0\leq s \leq t \leq T,
	\end{equation*}
then $(M_t)_{t \in [0,T]}$ is called a \emph{symmetric} $G$-martingale.
\end{defi}
{

It is shown in Theorem 6.2.5 in \cite{peng_nonlinearExpectation_book} that the $G$-expectation is an upper expectation, \blue{i.e.} there exists a weakly compact set of probability measures $\mathcal{P}\subseteq \mathcal{P}(\Omega)$ such that
\begin{equation} \label{eq:GExpectationUpperExpectation}
	\hat{\mathbb{E}}[X]= \sup_{P \in \mathcal{P}} {E}^P[X], \quad \text{ for each } X \in L_G^1(\Omega_T).
\end{equation}
Related to this set $\mathcal{P}$ we introduce the Choquet capacity defined by
\begin{equation}
	c(A):=\sup_{P \in \mathcal{P}} P(A), \quad A \in \mathcal{B}(\Omega_T). \label{eq:capacity}
\end{equation}
A set $A \subset \mathcal{B}(\Omega)$ is said to be ($\mathcal{P}$)-\emph{polar}, if $c(A)=0$. 
We say that a property holds \emph{quasi-surely} or for \emph{quasi all }$\omega \in \Omega$} if it holds outside a polar set.

\blue{
\begin{remark} \label{remark:ExplicitConstruction}
An explicit construction of the set of probability measures $\mathcal{P}$ in \eqref{eq:GExpectationUpperExpectation} is provided in Chapter 6.2.2 in \cite{peng_nonlinearExpectation_book}. Let $(\Omega, \mathcal{F}, P)$ be a probability space with a standard $d$-dimensional Brownian motion $W=(W_t)_{t \geq 0}$ and denote by $\mathbb{F}$ the natural filtration of $W$. For a given $\Theta \subseteq \mathbb{S}(d)$ as in \eqref{G-equationOneDimension}, call $\mathcal{A}_{0,\infty}^{\Theta}$ the set of all $\Theta$-valued $\mathbb{F}$-adapted processes on $[0,\infty)$. For each $\theta \in  \mathcal{A}_{0,\infty}^{\Theta}$  denote by $P^{\theta}$ the law of a stochastic integral $\int_0^{\cdot} \theta_s dW_s$. Next, define the sets 
\begin{align} \nonumber
\mathcal{P}_1:= \lbrace P^{\theta}: \theta \in  \mathcal{A}_{0,\infty}^{\Theta} \rbrace \quad \text{ and } \quad \mathcal{P}:= \overline{\mathcal{P}}_1,
\end{align}
where the closure is taken in the weak topology. In particular, $\mathcal{P}_1$ is tight, $\mathcal{P}$ is weakly compact and 
\greenNew{
\begin{equation*}
	\hat{\mathbb{E}}[X]= \sup_{P \in \mathcal{P}} {E}^P[X]= \sup_{P \in \mathcal{P}_1} {E}^P[X], \quad \text{ for each } X \in L_G^1(\Omega_T).
\end{equation*}}
\end{remark}
}

\section{Multi-dimensional fractional $G$-Brownian motion {($G$-fBm)}} \label{sec:Multi-DimensionalGfBm}
We now aim to define a multi-dimensional fractional Brownian motion in the framework of the $G$-expectation. In the literature, a one-dimensional fractional Brownian motion under volatility uncertainty {is introduced the first time in} \cite{fBrownianMotionGSetting1} and \cite{chen_2013}. We briefly discuss these two approaches. \\ In \cite{chen_2013} \blue{a one-dimensional} fractional $G$-Brownian motion {$\tilde{B}^H=(\tilde{B}^H)_{t \geq 0}$} with Hurst index $H \in (0,1)$ \blue{sublinear expectation space $(\Omega, \mathcal{H}, \hat{\mathbb{E}})$} is defined as a $G$-Gaussian process such that $\tilde{B}^H_0:=0$ and for all $t,s \geq 0$ it holds
\begin{align}
	\hat{\mathbb{E}}\left[\tilde{B}_t^H \tilde{B}^{H}_s\right]&=\frac{1}{2} \overline{\sigma}^2 \left( t^{2H}+s^{2H}- \vert t- s \vert^{2H} \right), \label{eq:CovarianceChen1} \\
	-\hat{\mathbb{E}}\left[-\tilde{B}_t^H \tilde{B}^{H}_s\right]&=\frac{1}{2} \underline{\sigma}^2 \left( t^{2H}+s^{2H}- \vert t- s \vert^{2H} \right)\blue{,}\label{eq:CovarianceChen2}
\end{align}
\blue{where $\underline{\sigma}^2:=-\mathbb{E}[-(\tilde{B}^H_1)^2]$ and $\overline{\sigma}^2:=\mathbb{E}[(\tilde{B}^H_1)^2]$.}
This definition corresponds to a straightforward {extension} of the definition of a classical fractional Brownian motion to the $G$-setting, but contains some mathematical problems, as also pointed out in \cite{fBrownianMotionGSetting1}. In particular, the definition of a $G$-Gaussian process in \cite{chen_2013} is not in line with the one introduced by Peng in \cite{peng_Gaussian_2011}, see Definition \ref{defi:G-Gaussian} and Remark \ref{remark:GBrownianMotionNotGaussian}. Furthermore, {the} moving average representation of \blue{a} fractional $G$-Brownian motion with respect to a two-sided $G$-Brownian motion {of} Theorem 1 in \cite{chen_2013}
{requires the existence of} a two-sided $G$-Brownian motion $B^{\textnormal{2-sided}}:=(B_{\blue{t}}^{\textnormal{2-sided}})_{t \in \mathbb{R}}$ given by \begin{align} \label{eq:DefinitionTwoSidedBm}
	B_t^{\textnormal{2-sided}}:=\begin{cases}
		B_t^{(1)} & \text{if} \quad t \geq 0,  \\
		B_{-t}^{(2)} &\text{if} \quad t <0,
\end{cases}
\end{align} where $B^{(1)}=(B^{(1)}_t)_{t \geq 0}$ and $B^{(2)}=(B^{(2)}_t)_{t \geq 0}$ are two independent $G$-Brownian motions, see Definition 9 in \cite{chen_2013}. To the best of our knowledge {the} definition of {independent stochastic processes} in the $G$-setting {is very delicate, as it is not mutual.} By Definition \ref{defi:independenceRandomVariables} we have to distinguish if a random vector $X$ is independent of \blue{a} random vector $Y$ or if $Y$ is independent of $X$. In general, {the fact that} $X$ is independent of $Y$ {does not imply} that $Y$ is independent of $X$, see e.g. Example 1.3.15 in \cite{peng_nonlinearExpectation_book} for a counterexample. {This extends to the following definition of distributional independence.
\begin{defi} \label{defi:distributionallyIndependence}
Let $(X_t)_{t \geq 0}$ and $(Y_t)_{t \geq 0}$ be two stochastic processes \blue{on} a nonlinear expectation space {$(\Omega, \mathcal{H}, \mathbb{E})$}. \greenNew{We say that $(Y_t)_{t \geq 0}$ is} \emph{distributionally independent} of $(X_t)_{t \geq 0}$ if, for each $\underline{t}=(t_1,...,t_n) \in \mathcal{T}$, $(Y_{t_1},...,Y_{t_n})$ is independent of $(X_{t_1},...,X_{t_n})$, where $\mathcal{T}$ is defined in \eqref{eq:VectorTimes}.
\end{defi}
} 
One might interpret the independence of $B^{(1)}=(B_t^{(1)})_{t \geq 0}$ and $B^{(2)}=(B_t^{(2)})_{t \geq 0}$ in the sense that $B^{(1)}$ is distributionally independent of $B^{(2)}$ and vice versa, i.e. $B^{(2)}$ is distributionally independent of $B^{(1)}$. However, this is not possible as we prove next.
		\begin{lemma} \label{lemma:IndependentGBrownianMotion}
			Let $(\Omega, \mathcal{H}, \hat{\mathbb{E}})$ be a sublinear expectation space \greenNew{endowed} with two one-dimensional $G$-Brownian motions $B^{(1)}=(B_t^{(1)})_{t \geq 0}$ and $B^{(2)}=(B_t^{(2)})_{t \geq 0}$ \blue{such that \linebreak $ -\hat{\mathbb{E}}[-(B^{(1)}_1)^2]=-\hat{\mathbb{E}}[-(B^{(2)}_1)^2]=\underline{\sigma}^2 \neq \overline{\sigma}^2= \hat{\mathbb{E}}[(B^{(1)}_1)^2]= \hat{\mathbb{E}}[(B^{(2)}_1)^2]$}. Then it is not possible that $B^{(1)}$ is distributionally independent of $B^{(2)}$ and $B^{(2)}$ is distributionally independent of $B^{(1)}$.
		\end{lemma}
		\begin{proof}
			Let $B^{(1)}=(B_t^{(1)})_{t \geq 0}$ and $B^{(2)}=(B_t^{(2)})_{t \geq 0}$ be two one-dimensional $G$-Brownian motions on the sublinear expectation space {$(\Omega, \mathcal{H}, \hat{\mathbb{E}})$}. Assume now that $B^{(1)}$ is distributionally independent of $B^{(2)}$ and $B^{(2)}$ is distributionally independent of $B^{(1)}$. Fix $t >0$. Then, by Definition \ref{defi:distributionallyIndependence} the random variable $B_t^{(1)}$ is independent of the random variable $B_t^{(2)}$ and {$B_t^{(2)}$ is independent of $B_t^{(1)}$}, both in the sense of Definition \ref{defi:independenceRandomVariables}. Moreover, $B_t^{(1)}$ as well as $B_t^{(2)}$ are $G$-normally distributed. However, this is a contradiction to Theorem 15 in \cite{hu_li_2012} which states that $B_t^{(1)}$ and $B_t^{(2)}$ need to be maximally distributed in the sense of Definition \ref{defi:G-MaximalDistribution}, whenever $B_t^{(1)}$ is independent of $B_t^{(2)}$ as well as $B_t^{(2)}$ is independent of $B_t^{(1)}$. 		\end{proof}
		\begin{remark} \label{remark:IndependenceGBm}
		Lemma \ref{lemma:IndependentGBrownianMotion} is also in line with \cite{peng_multidimensional_2008}, which {states} that it is in general not possible to find a system of coordinates under which the corresponding components $B(i), i=1,...,d{,}$ of a $d$-dimensional $G$-Brownian motion $B=(B(1),...,B(d))$ are mutually independent \blue{of} each other.
		\end{remark}
{A solution to this problem} is provided in Definition 4 in \cite{fBrownianMotionGSetting1}. 
\begin{defi}\label{def:fBrownianMotionSetting1}
	Let $H \in (0,1)$. Then a continuous stochastic process $(\tilde{B}^H_t)_{t \geq 0}$ on a sublinear expectation space {$(\Omega, \mathcal{H}, \hat{\mathbb{E}})$} is called a \emph{one-dimensional fractional $G$-Brownian motion} \blue{of} Hurst \blue{index} $H$, if 
	\begin{enumerate}
		\item $\tilde{B}^H_0=0$, and for all $t \geq 0$ it holds $ 
			-\hat{\mathbb{E}}\left[-\tilde{B}^H_t\right]=\hat{\mathbb{E}}\left[\tilde{B}^{H}_t\right]=0$\blue{;}
		\item For all $t,s \geq 0$, \blue{equations \eqref{eq:CovarianceChen1} and \eqref{eq:CovarianceChen2} hold with}
		with $\underline{\sigma}^2=-\hat{\mathbb{E}}\left[-\tilde{B}^{H}_1\right]$ and $\overline{\sigma}^2=\hat{\mathbb{E}}\left[\tilde{B}^{H}_1\right]$\blue{;}
		\item For each $t,s \geq 0, \ \tilde{B}^{H}_{t+s}-\tilde{B}^{H}_s$ and $\tilde{B}^{H}_t$ are identically distributed.
		\item $\lim_{t \to 0} \hat{\mathbb{E}} \left[\vert \tilde{B}^{H  }_t\vert^m\right]t^{-2H}=0$ for each $m \in \mathbb{N}$ and $m \geq 3$.
	\end{enumerate}
\end{defi}
{T}he existence of such a process given in Definition \ref{def:fBrownianMotionSetting1} is proved in Theorem 2 in \cite{fBrownianMotionGSetting1} {choosing $\Omega=C_0(\mathbb{R}_+,\mathbb{R})$ and $\mathcal{H}=Lip(\Omega)$.} {The authors also derive a representation theorem of \blue{a} fractional $G$-Brownian motion as an integral of a stochastic kernel and a volatility function with respect to a standard Brownian motion.}  However, it is mentioned in \cite{fBrownianMotionGSetting1} that the studied representation is not unique, as \blue{a} classical fractional Brownian motion can be represented in at least three different ways as an integral with respect to a standard Brownian motion.
Besides this\blue{,} it is shown that the fractional $G$-Brownian motion in Definition \ref{def:fBrownianMotionSetting1} satisfies properties such as self-similarity and long-range dependence, see \cite{fBrownianMotionGSetting1} for more details. \\
\\
{Our} goal is now to introduce a mathematically rigorous definition of a multi-dimensional fractional $G$-Brownian motion. At first sight, this seems to be straightforward by {extending} Definition \ref{def:fBrownianMotionSetting1}. However, some further considerations {are} necessary for the multi-dimensional case. First, by Remark \ref{remark:IndependenceGBm} the components of a multi-dimensional $G$-Brownian motion are no any longer independent, in contrast to the ones of a standard Brownian motion, {as we can see by looking at the quadratic variation process of a $d$-dimensional $G$-Brownian motion, see \eqref{eq:QuadraticVariationD}-\eqref{eq:QuadraticVariationD_new}.} Thus, the components of a multi-dimensional fractional $G$-Brownian motion {cannot be assumed to be}  independent, as in the classical case. Therefore, we need to take into account the covariance structure of the components of \blue{a} multi-dimensional fractional $G$-Brownian which stems from the covariance structure of a multi-dimensional $G$-Brownian motion. In order to capture this property we define a multi-dimensional fractional $G$-Brownian motion as an integral with respect to a multi-dimensional $G$-Brownian motion, by generalizing the Volterra representation of a standard fractional Brownian motion, see {Proposition 2.5 of \cite{nourdin_fractional_brownian_motion}}. The choice for the Volterra representation is particularly motivated as the representation with a two-sided Brownian motion yields several mathematical difficulties as discussed before, see Lemma \ref{lemma:IndependentGBrownianMotion}. Moreover, at the end of this section we show that Definition {\ref{def:FractionalGBrownianMotionOurDDimension}} of a $d$-dimensional fractional $G$-Brownian motion is in line with Definition \ref{def:fBrownianMotionSetting1} introduced in \cite{fBrownianMotionGSetting1}, see Lemma \ref{lemma:DefinitionCorrespondsD1}. 

\subsection{Definition of a \blue{multi}-dimensional {$G$-fBm} and first properties}

{From now on, we {set} $\Omega:=C_0(\mathbb{R}_+,\mathbb{R}^d)$, for $d \in \mathbb{N}$, and $\mathcal{H}:=Lip(\Omega)$. We equip $\Omega$ with the Borel $\sigma$-algebra $\mathcal{F}:=\mathcal{B}(\Omega)$. For these choices, it is indeed possible to construct a $d$-dimensional $G$-Brownian motion and the $G$-expectation {$\hat{\mathbb{E}}$} as shown in Section \ref{sec:Gsetting}, along with the weakly compact set of probability measures $\mathcal{P}$ satisfying
\eqref{eq:GExpectationUpperExpectation}.
}
{Then}, let $B=(B_t)_{t \geq 0}$ be {the canonical }$d$-dimensional $G$-Brownian motion {on $(\Omega, \mathcal{H}, \hat{\mathbb{E}})$}. 

\begin{defi}\label{def:FractionalGBrownianMotionOurDDimension}
		A \emph{$d$-dimensional fractional $G$-Brownian motion {($G$-fBm)} of Hurst index $H \in \left( 0,1\right)$ with $B^H(i)=(B_t^H(i))_{t \geq 0}$} for $i=1,...,d$ is defined as follows. For $H \in \left( 0,\frac{1}{2}\right) \cup \left( \frac{1}{2},1\right)$ we set \footnote{{Formula \eqref{eq:FractionalGBrownianMotionOurDefintionD} has been given for a fractional Brownian motion in the classical setting in Theorem 5.2 of \cite{norros1999anelementary}.}}
		\begin{equation} \label{eq:FractionalGBrownianMotionOurDefintionD}
		B_t^H(i):=\int_0^t K_H(t,s) dB_s(i),	
		\end{equation}
		where $B=(B(i))_{i=1,...\blue{,}d}$ with $B(i)=(B_t(i))_{t \geq 0}$ for $i=1,...,d$ is a $d$-dimensional $G$-Brownian motion. Here, we set for $t > s >0$ 
	\begin{align}
		c_H&:=\sqrt{\frac{H(2H-1)}{\int_0^1 (1-x)^{1-2H} x^{H-\frac{3}{2}} dx}}  &\text{for }  H>\frac{1}{2}, \label{eq:VolterraConstant1}\\
		d_H&:=\sqrt{\frac{2H}{(1-2H)\int_0^1 (1-x)^{-2H} x^{H-\frac{1}{2}} dx}}  &\text{for }  H<\frac{1}{2}, \label{eq:VolterraConstant2}
	\end{align}
	and 
	\small{\begin{align} \label{eq:KernelVolterra}
	K_H(t,s):=\begin{cases}
		c_H s^{\frac{1}{2}-H} \int_s^t (u-s)^{H-\frac{3}{2}} u^{H-\frac{1}{2}}du& \text{for} \quad H > \frac{1}{2}  \\
			d_H \left[ \left(\frac{t}{s}\right)^{H-\frac{1}{2}}(t-s)^{H-\frac{1}{2}} - \left(H-\frac{1}{2}\right)s^{\frac{1}{2}-H} \int_{s}^t u^{H-\frac{3}{2}}(u-s)^{H-\frac{1}{2}}\right]du &\text{for} \quad H <\frac{1}{2}.
		\end{cases}	
	\end{align}}
		For {the} Hurst index $H=\frac{1}{2}$, we define a fractional $G$-Brownian motion $B^{\frac{1}{2}}=(B^{\frac{1}{2}}_t)_{t \geq 0}$ by
		\begin{equation}  \label{eq:FractionalGBrownianMotionOurDefintionD2}
			B_t^{\frac{1}{2}}:=B_t,
		\end{equation} 
		where $B=(B_t)_{t \geq 0}$ is a $d$-dimensional $G$-Brownian motion.
	\end{defi}	

The integral in \eqref{eq:FractionalGBrownianMotionOurDefintionD} has to be understood in the sense of equation \eqref{eq:DefinitionItoScalarProduct}. More specifically, we denote by $(\textbf{e}_i)_{i=1,...,d}$ the canonical vector basis of $\mathbb{R}^d$, i.e. $\textbf{e}_i=(\textbf{e}_i(j))_{j=1,...,d}$ with $\textbf{e}_i(j)=\delta_i(j)$, where $\delta_i$ denotes the Dirac delta for $i$. By  \eqref{eq:DefinitonScalarGBrownianMotion} we have 
\begin{equation} 
B_t(i)=B_t^{\textbf{e}_i}, \quad t \geq 0, \quad i=1,...,d.	
\end{equation}
Moreover, by \eqref{eq:FunctionGScalar} and \eqref{eq:DefinitionBoundariesVolatilityScalar} we get that $B(i)$ is a one-dimensional $G_i$-Brownian motion with
\begin{equation} \label{eq:DefinitionGComponent}
	G_{i}(\alpha):=G_{\textbf{e}_i \textbf{e}_i^T}= \frac{1}{2}\left( \overline{\sigma}^2_{i} \alpha^+ - \underline{\sigma}^2_{i}\alpha^-  \right), \quad \alpha \in \mathbb{R}^d,
\end{equation}
where 
\begin{align} \label{eq:DefinitionGComponentBoundaries}
	\overline{\sigma}^2_i:= \sigma^2_{\textbf{e}_i\textbf{e}_i^T}= \sup_{\gamma \in \Theta} \gamma_{ii} \quad \text{ and } \quad \underline{\sigma}^2_{i}:= \sigma^2_{\textbf{e}_i\textbf{e}_i^T}= \inf_{\gamma \in \Theta} \gamma_{ii}.
\end{align}
Here, $\Theta \subseteq \mathbb{S}(d)$ is given in \eqref{G-equationOneDimension} and $\gamma=(\gamma_{ij})_{i,j=1,...,d}$.
\begin{remark} \label{remark:IntegralWellDefinedDimension}
Note that \blue{for fixed $t \in [0,T]$} the $G$-It\^{o} integral in \eqref{eq:FractionalGBrownianMotionOurDefintionD} is well-defined as $K_H=(K_H(t,s))_{{0 \le s \le t}} \in M_G^2(0,T)$ for $H \in \left(0,\frac{1}{2}\right) \cup \left( \frac{1}{2},1\right)$. This holds as $K_H$ is deterministic and
\begin{align*}
	\| K_H(t,s) \|_{M_G^2(0,T)} =\left( \hat{\mathbb{E}}\left[ \int_0^T (K_H(t,s))^2ds\right]\right)^{\frac{1}{2}}=
	\left( \int_0^T (K_H(t,s))^2ds\right)^{\frac{1}{2}}<\infty, 
\end{align*}
for any $0 \le s \le t,$ see \blue{Proposition 5.1.3} in \cite{nualart_20006} and Proposition 2.5 in \cite{nourdin_fractional_brownian_motion}. Thus, by Lemma 3.3.4 in \cite{peng_nonlinearExpectation_book}, \blue{$B^H_t \in L_G^2(\Omega_T)$ for every $t \in [0,T]$}.\end{remark}

\blue{We now prove that the $G$-fBm introduced in Definition \ref{def:FractionalGBrownianMotionOurDDimension} satisfies componentwise equations \eqref{eq:CovarianceChen1}-\eqref{eq:CovarianceChen2}. This is a generalization of the well-known covariance property of \blue{a} classical fBm in the $G$-setting. }

\begin{prop} \label{Prop:CovariationComponentwise}
Let $B^H=(B^H_t)_{t \geq 0}$ be a $d$-dimensional fractional $G$-Brownian motion of Hurst index $H \in \left( 0,1\right)$. Let $i=1,...,d.$ It holds $B_0^H(i)=0$, and for each {$t \geq 0$} 
\begin{equation}\label{eq:GfBmCenteredDdimension}
	\hat{\mathbb{E}}\left[B^H_t(i)\right]=-\hat{\mathbb{E}}\left[-B^H_t(i)\right]=0.
\end{equation}
Moreover, for {$s,t>0$} we have
	\begin{align}
			\hat{\mathbb{E}}\left[B^H_s(i)B^H_t(i)\right]=\frac{1}{2} \overline{\sigma}^2_i \left( t^{2H} + s ^{2H}-\vert t -s\vert^{2H}\right), \label{eq:VarianceStructureDDimension1}\\
			-\hat{\mathbb{E}}\left[-B_s^H(i)B_t^H(i)\right]=\frac{1}{2} \underline{\sigma}_i^2 \left( t ^{2H} + s ^{2H}-\vert t -s\vert^{2H}\right), \label{eq:VarianceStructureDDimension2} 
	\end{align}
where $\overline{\sigma}^2_i$ and $\underline{\sigma}^2_i$ are given in \eqref{eq:DefinitionGComponentBoundaries}.
\end{prop}

\begin{proof}
		Let $i=1,...,d$, $0<s\leq t$ and $H\in \left( 0,\frac{1}{2}\right) \cup \left(\frac{1}{2},1\right)$. Equation \eqref{eq:GfBmCenteredDdimension} follows directly by page 60 in \cite{peng_nonlinearExpectation_book}. Moreover, we have
	\small{\begin{align}
		\hat{\mathbb{E}}\left[ B_t^H(i) B_s^H(i) \right] 
		&=\hat{\mathbb{E}}\left[ \left(\int_0^t K_H(t,u) dB_u(i) \right) \left(\int_0^s K_H(s,u) dB_u(i) \right) \right] \nonumber  \\
		&=\hat{\mathbb{E}}\left[ \left(\int_0^s K_H(t,u) dB_u(i) + \int_s^t K_H(t,u) dB_u(i)   \right) \left(\int_0^s K_H(s,u) dB_u(i) \right) \right]\nonumber \\
		&=\hat{\mathbb{E}}\bigg[ \left(\int_0^s K_H(t,u) dB_u(i)\right)\left(\int_0^s K_H(s,u) dB_u(i) \right) \notag \\
		&\qquad+ \left(\int_s^t K_H(t,u) dB_u(i)   \right) \left(\int_0^s K_H(s,u) dB_u(i) \right) \bigg] \nonumber \\
		&=\hat{\mathbb{E}}\bigg[ \hat{\mathbb{E}} \bigg[\left(\int_0^s K_H(t,u) dB_u(i)\right)\left(\int_0^s K_H(s,u) dB_u(i) \right) \notag \\ & \qquad + \left(\int_s^t K_H(t,u) dB_u(i)   \right) \left(\int_0^s K_H(s,u) dB_u(i) \right) \bigg \vert \Omega_s \bigg]\bigg], \label{eq:VarianceCalculuation1}
	\end{align}}
	where we use Proposition 3.2.3(v) in \cite{peng_nonlinearExpectation_book} in \eqref{eq:VarianceCalculuation1}.
	Call $X^t=(X^t_s)_{\blue{0 \le s \le t}}$ the process defined by    
	$$
	X^t_s(i):=\int_0^s K_H(t,u) dB_u(i), \quad \blue{0 \le s \le t}.
	$$
	Note that the random variable $X^t_s(i)X^s_s(i)$ is $\Omega_s$-measurable as $K_H(s,u), K_H(t,u)$ are deterministic. Therefore, by \eqref{eq:VarianceCalculuation1} we get 
	\begin{align}
		\hat{\mathbb{E}}\left[ B_t^H(i) B_s^H(i) \right] &=\hat{\mathbb{E}}\left[ X^t_s(i)X^s_s(i)+\hat{\mathbb{E}} \left[ \left(\int_s^t K_H(t,u) dB_u(i)   \right) \left(\int_0^s K_H(s,u) dB_u(i) \right) \bigg \vert \Omega_s \right]\right] \label{eq:VarianceCalculation2} \\
		&=\hat{\mathbb{E}}\bigg[ X^t_s(i)X^s_s(i)+\left(\int_0^s K_H(s,u) dB_u(i) \right)^+\hat{\mathbb{E}} \left[ \int_s^t K_H(t,u) dB_u(i)    \bigg \vert \Omega_s \right] \nonumber  \\
		&\quad \quad +\left(\int_0^s K_H(s,u) dB_u(i) \right)^-\hat{\mathbb{E}} \left[ -\int_s^t K_H(t,u) dB_u(i)    \bigg \vert \Omega_s \right]\bigg] \label{eq:VarianceCalculation3} \\
		&=\hat{\mathbb{E}}\left[X^t_s(i)X^s_s(i)\right]. \label{eq:VarianceCalculation4}
	\end{align}
	Here, \eqref{eq:VarianceCalculation2} follows by Remark 3.2.4 in \cite{peng_nonlinearExpectation_book}. Moreover, we use Proposition 3.2.3(iv) in \cite{peng_nonlinearExpectation_book} in \eqref{eq:VarianceCalculation3} and Proposition 3.3.6(iii) in \cite{peng_nonlinearExpectation_book} in \eqref{eq:VarianceCalculation4}. 
	By applying $G$-It\^{o}'s formula given by Theorem 3.6.3 in \cite{peng_nonlinearExpectation_book} it holds
	\begin{align}
		{X^t_s(i)X^s_s(i)}&=\int_0^s {X^s_u(i)}K_H(t,u) dB_u(i)+ \int_0^s {X^t_u(i)}K_H(s,u) dB_u(i) \nonumber \\
		& \quad + \int_0^s K_H(s,u) K_H(t,u) d\langle B(i) \rangle_u. \label{eq:VarianceCalculation5}
	\end{align}
	Thus, \eqref{eq:VarianceCalculation4} together with \eqref{eq:VarianceCalculation5} implies
	\begin{align}
		\hat{\mathbb{E}}\left[ B_t^H(i) B_s^H(i) \right]&=  \hat{\mathbb{E}}\left[\int_0^s K_H(s,u) K_H(t,u) d\langle B(i) \rangle_u \right] \label{eq:VarianceCalculation6} \\
		&=   \hat{\mathbb{E}}\left[\int_0^s K_H(s,u) K_H(t,u) d\langle B(i) \rangle_u \pm 2\int_0^s G_{i} \left( K_H(s,u) K_H(t,u)\right)du \right] \nonumber \\
		&=   \hat{\mathbb{E}}\left[\int_0^s K_H(s,u) K_H(t,u) d\langle B (i)\rangle_u - 2\int_0^s G_{i} \left( K_H(s,u) K_H(t,u)\right)du \right] \nonumber \\
		&\quad  +2\int_0^s G_{i} \left( K_H(s,u) K_H(t,u)\right)du \nonumber \\
		&=2\int_0^s G_{i} \left( K_H(s,u) K_H(t,u)\right)du \label{eq:VarianceCalculation7} \\
		&= \overline{\sigma}^2_i \int_0^s K_H(s,u)K_H(t,u)du. \label{eq:VarianceCalculation8} 
		\end{align}
	\greenNew{Here and similarly in the following, the notation $\pm2\int_0^s G_{i} \left( K_H(s,u) K_H(t,u)\right)du$ indicates that the term $2\int_0^s G_{i} \left( K_H(s,u) K_H(t,u)\right)du$ is added and subtracted simultaneously.}
		%\greenNew{Here and in the following,  $\pm x := +x-x$ for given $x \in \mathbb{R}$.}
	In \eqref{eq:VarianceCalculation6} we use Lemma 3.3.4 in \cite{peng_nonlinearExpectation_book} together with the fact that for $X,Y \in L_G^1(\Omega)$ such that $\hat{\mathbb{E}}[X]=-\hat{\mathbb{E}}[-X]=0$, it holds $\hat{\mathbb{E}}[X+Y]=\hat{\mathbb{E}}[Y]$. Moreover, \eqref{eq:VarianceCalculation7} follows as $\overline{M}=(\overline{M}_t)_{t \geq 0}$ defined by 
	\begin{equation} \label{eq:MartingaleProperty}
		\overline{M}_t=\int_0^t \eta_u d \langle B(i) \rangle_u - 2 \int_0^t G_i(\eta_u)du,
	\end{equation}
	for $t \geq 0$ and $\eta \in M_G^1(0,T; \mathbb{R})$, is a $G$-martingale with $\blue{\overline{M}}_0=0$. \\
	\allowdisplaybreaks
	By doing the same calculations for $-B_t^HB_s^H$ it follows for $0<s<t$
	\begin{align}
		\hat{\mathbb{E}}\left[- B_t^H(i) B_s^H(i) \right]&=  \hat{\mathbb{E}}\left[\int_0^s -K_H(s,u) K_H(t,u) d\langle B(i) \rangle_u \right] \nonumber \\
		&=2\int_0^s G_i(-K_H(s,u) K_H(t,u))du \nonumber \\
		&=- \underline{\sigma}^2_i \int_0^s K_H(s,u)K_H(t,u)du. \label{eq:VarianceCalculation9} 
	\end{align}
	By \blue{Proposition 5.1.3} in \cite{nualart_20006} and Proposition 2.5 in \cite{nourdin_fractional_brownian_motion} it holds
	\begin{equation} \label{eq:IntegralKernels}
		\int_0^sK_H(s,u)K_H(t,u)du=\frac{1}{2} \left( t^{2H}+s^{2H}-\vert t-s \vert^{2H} \right).
	\end{equation}
	By putting together \eqref{eq:VarianceCalculation8} and \eqref{eq:VarianceCalculation9} with \eqref{eq:IntegralKernels}, equations \eqref{eq:VarianceStructureDDimension1} and \eqref{eq:VarianceStructureDDimension2} follow for \\$H \in \left(0,\frac{1}{2} \right) \cup \left( \frac{1}{2}, 1 \right)$. 
	Let now $H=\frac{1}{2}$ and $0<t=s$. Then it holds			
	\begin{align*}
		\hat{\mathbb{E}}\left[ B_{t}(i) B_t(i) \right]=\hat{\mathbb{E}}\left[ \left( \int_0^t dB_u(i) \right)^2 \right]= \hat{\mathbb{E}}\left[\int_0^t d\langle B(i) \rangle_u  \right]=\hat{\mathbb{E}}\left[\langle B(i) \rangle_t  \right]=\overline{\sigma}^2_it,
	\end{align*}
	where we use for the second equality the $G$-It\^{o}-isometry, see Proposition 3.4.5 in \cite{peng_nonlinearExpectation_book}. In the last equality we apply Lemma 3.4.1 in \cite{peng_nonlinearExpectation_book}. \\
	By the definition of the quadratic variation process of a one-dimensional $G$-Brownian motion as $\langle B(i) \rangle_t:=(B_t(i))^2-2\int_0^t B_u(i) dB_u(i)$, $t \ge 0$, we get that 
		\begin{equation}
			\hat{\mathbb{E}}[-\langle B(i) \rangle_t ]= \hat{\mathbb{E}}\left[- (B_t(i))^2 + 2 \int_0^t B_u(i) {dB_u(i)} \right]= \hat{\mathbb{E}}\left[- (B_t(i))^2 \right],
		\end{equation}
	by Proposition 3.3.6(iii) in \cite{peng_nonlinearExpectation_book}. Together with Lemma 3.4.1 in \cite{peng_nonlinearExpectation_book}, this implies that $-\hat{\mathbb{E}}\left[- (B_t(i))^2 \right]= \underline{\sigma}_i^2t$. We now consider $0<s<t$. Then we have
	\begin{align}
		\hat{\mathbb{E}}[B_t(i) B_s(i)]&= \hat{\mathbb{E}}\left[ B_t(i)B_s(i) -(B_s(i))^2 + (B_s(i))^2\right] \nonumber \\
		&=\hat{\mathbb{E}}\left[B_s(i)(B_t(i)-B_s(i))+(B_s(i))^2\right] \nonumber  \\
		&=\hat{\mathbb{E}}\left[\hat{\mathbb{E}}\left[B_s(i)(B_t(i)-B_s(i))+(B_s(i))^2 \vert \Omega_s \right] \right] \label{eq:CovarianceGBM1} \\
		&=\hat{\mathbb{E}}\left[\hat{\mathbb{E}}\left[B_s(i)(B_t(i)-B_s(i)) \vert \Omega_s \right] +(B_s(i))^2 \right] \label{eq:CovarianceGBM2} \\
		&=\hat{\mathbb{E}}\left[(B_s(i))^+\hat{\mathbb{E}}\left[B_t(i)-B_s(i) \vert \Omega_s \right]+(B_s(i))^-\hat{\mathbb{E}}\left[-(B_t(i)-B_s(i)) \vert \Omega_s \right]  +(B_s(i))^2 \right] \label{eq:CovarianceGBM3} \\
		&=\hat{\mathbb{E}} \left[ (B_s(i))^2 \right] \label{eq:CovarianceGBM4} \\
		&= \overline{\sigma}_i^2s. \nonumber
	\end{align}
	Here, we use Proposition 3.2.3(v) in \cite{peng_nonlinearExpectation_book} in \eqref{eq:CovarianceGBM1} and Remark 3.2.4 in \cite{peng_nonlinearExpectation_book} in \eqref{eq:CovarianceGBM2}. Moreover, \eqref{eq:CovarianceGBM3} follows by Proposition 3.2.3(iv) in \cite{peng_nonlinearExpectation_book}. Finally, we get \eqref{eq:CovarianceGBM4} by Example 3.2.10 in \cite{peng_nonlinearExpectation_book}. By the same arguments it follows
	\begin{align}
		\blue{-}\hat{\mathbb{E}}[-B_t(i) B_s(i)] &=\blue{-} \hat{\mathbb{E}}\left[-B_s(i)(B_t(i)-B_s(i))-(B_s(i))^2\right] \nonumber  \\
		&=\blue{-}\hat{\mathbb{E}}\left[\hat{\mathbb{E}}\left[-B_s(i)(B_t(i)-B_s(i)) \vert \Omega_s \right] -(B_s(i))^2 \right] \nonumber\\
		&=\blue{-}\hat{\mathbb{E}}\left[(-B_s(i))^+\hat{\mathbb{E}}\left[B_t(i)-B_s(i) \vert \Omega_s \right]+(-B_s(i))^-\hat{\mathbb{E}}\left[-(B_t(i)-B_s(i)) \vert \Omega_s \right]  -(B_s(i))^2 \right] \nonumber \\
		&=\blue{-}\hat{\mathbb{E}} \left[ -(B_s(i))^2 \right] \nonumber \\
		&= \underline{\sigma}_i^2s. \nonumber
	\end{align}
	\end{proof}
As the components of a $G$-Brownian motion are not independent, the components of \blue{a} fractional $G$-Brownian motion do not satisfy this property either. \\
For $i=1,...,d$ and $i \neq j$ we define
\begin{equation}
	\overline{\sigma}^2_{ij}:=\sup_{\gamma \in \Theta} (\gamma_{ii}+ \gamma_{jj}+2\gamma_{ij})-   \inf_{\gamma \in \Theta} (\gamma_{ii}+ \gamma_{jj}-2\gamma_{ij}),
\end{equation}
and
\begin{equation}
	\underline{\sigma}^2_{ij}:=-\sup_{\gamma \in \Theta} (\gamma_{ii}+ \gamma_{jj}-2\gamma_{ij})+\inf_{\gamma \in \Theta} (\gamma_{ii}+ \gamma_{jj}+2\gamma_{ij}),
\end{equation}
\greenNew{for $\Theta \subseteq \mathbb{S}(d)$ as in \eqref{G-equationOneDimension}. }
\begin{prop}\label{prop:sitj}
Let $B^H=(B^H_t)_{t \geq 0}$ be a $d$-dimensional \blue{fractional} $G$-Brownian motion of Hurst index $\blue{H \in (0,1)}$. Let $i,j=1,...,d$, $i \neq j$ and $s,t \geq 0$. Then it holds \blue{for $H \in (0,\frac{1}{2}) \cup (\frac{1}{2},1)$}
\begin{align}
	\hat{\mathbb{E}} \left[ B_s^H(i) B_t^H(j) \right] & \leq  \frac{1}{8} \left(t^{2H}+s^{2H}-\vert t-s \vert^{2H} \right) \overline{\sigma}_{ij}^2 \nonumber \\
	-\hat{\mathbb{E}} \left[ -B_s^H(i) B_{\blue{t}}^H(j) \right]& \geq   \frac{1}{8} \left(t^{2H}+s^{2H}-\vert t-s \vert^{2H} \right) \underline{\sigma}_{ij}^2 .\nonumber
\end{align}
\blue{
For $H = \frac{1}{2}$ we have
\begin{align}
	\hat{\mathbb{E}} \left[ B_s^H(i) B_t^H(j) \right] & =(s \wedge t)  \sup_{\gamma \in \Theta} \gamma_{ij} \nonumber \\
	-\hat{\mathbb{E}} \left[ -B_s^H(i) B_t^H(j) \right]& = (s \wedge t) \inf_{\gamma \in \Theta} \gamma_{ij},\label{eq:secondforH12sitj}
\end{align}
for given $\Theta \subseteq \mathbb{S}(d)$ in \eqref{G-equationOneDimension}. }
\end{prop}
\begin{proof}
\blue{Assume first $H \in (0,\frac{1}{2}) \cup (\frac{1}{2},1)$.}
By similar calculations as in the proof of Proposition \ref{Prop:CovariationComponentwise} it follows that
\begin{align}
	\hat{\mathbb{E}} \left[ B_s^H(i) B_t^H(j) \right]
	&= \hat{ \mathbb{E}} \left[ \int_0^s K_H(s,u) K_H(t,u) d \langle B(i),B(j) \rangle_u \right] \nonumber \\
	&= \frac{1}{4} \hat{\mathbb{E}} \left[ \int_0^s K_H(s,u) K_H(t,u) d \langle B^{\textbf{e}_i+ \textbf{e}_j} \rangle_u + \int_0^s -K_H(s,u) K_H(t,u)  d \langle B^{\textbf{e}_i- \textbf{e}_j} \rangle_u \right],\label{eq:ExpectationIncrements1}
\end{align}
where we use \eqref{eq:CovariationMatrixGBrownianmotionD}. Note that $B^{\textbf{e}_i+ \textbf{e}_j}$ is a one-dimensional $G_{\textbf{e}_i+ \textbf{e}_j}$-Brownian motion with 
\begin{align}
	G_{\textbf{e}_i+ \textbf{e}_j}(\alpha) &=\frac{1}{2}\left(2{G}\left((\textbf{e}_i+ \textbf{e}_j)(\textbf{e}_i+ \textbf{e}_j)^T\right) \alpha^+ + 2G\left(-(\textbf{e}_i+ \textbf{e}_j)(\textbf{e}_i+ \textbf{e}_j)^T \right) \alpha^-\right) \nonumber \\
	&={\frac{1}{2}\left(\sup_{\gamma \in \Theta} \textnormal{tr}\left[ (\textbf{e}_i+ \textbf{e}_j)(\textbf{e}_i+ \textbf{e}_j)^T \gamma\right] \alpha^+ {+} \sup_{\gamma \in \Theta}\textnormal{tr}\left[-(\textbf{e}_i+ \textbf{e}_j)(\textbf{e}_i+ \textbf{e}_j)^T \gamma \right] \alpha^-\right)} \nonumber \\
	&=\frac{1}{2}\left(\sup_{\gamma \in \Theta}(\gamma_{ii}+ \gamma_{jj}+2\gamma_{ij})\alpha^+{+} \sup_{\gamma \in \Theta}\left(-\gamma_{ii}-{\gamma_{jj}}-2\gamma_{ij}\right) \alpha^-\right) \nonumber \\
	&={\frac{1}{2}} \left( \sup_{\gamma \in \Theta} (\gamma_{ii}+ \gamma_{jj}+2\gamma_{ij})\alpha^+ - \inf_{\gamma \in \Theta} (\gamma_{ii}+ \gamma_{jj}+2\gamma_{ij})\alpha^- \right)\label{eq:ExpectationIncrements2}
\end{align}
and
$B^{\textbf{e}_i- \textbf{e}_j}$ is a one-dimensional $G_{\textbf{e}_i- \textbf{e}_j}$-Brownian motion with 
\begin{align}
	G_{\textbf{e}_i- \textbf{e}_j}(\alpha) &=\frac{1}{2}\left({2}{G}\left((\textbf{e}_i- \textbf{e}_j)(\textbf{e}_i- \textbf{e}_j)^T)\right) \alpha^+ + {2}{G}\left(-(\textbf{e}_i- \textbf{e}_j)(\textbf{e}_i-\textbf{e}_j)^T\right) \alpha^- \right) \nonumber \\
	&={\frac{1}{2}\left(\sup_{\gamma \in \Theta} \textnormal{tr}\left[ (\textbf{e}_i- \textbf{e}_j)(\textbf{e}_i- \textbf{e}_j)^T \gamma\right] \alpha^+ {+} \sup_{\gamma \in \Theta}\textnormal{tr}\left[-(\textbf{e}_i- \textbf{e}_j)(\textbf{e}_i- \textbf{e}_j)^T \gamma \right] \alpha^-\right)} \nonumber \\
	&={\frac{1}{2}\left(\sup_{\gamma \in \Theta}(\gamma_{ii}+ \gamma_{jj}-2\gamma_{ij})\alpha^+{+} \sup_{\gamma \in \Theta}\left(-\gamma_{ii}-\gamma_{jj}+2\gamma_{ij}\right) \alpha^-\right)} \nonumber \\
	&= {\frac{1}{2}} \left( \sup_{\gamma \in \Theta} (\gamma_{ii}+ \gamma_{jj}-2\gamma_{ij})\alpha^+ - \inf_{\gamma \in \Theta} (\gamma_{ii}+ \gamma_{jj}-2\gamma_{ij})\alpha^-\right). \label{eq:ExpectationIncrements3}
\end{align}
By \blue{using the sublinearity of $\hat{\mathbb{E}}$} and \blue{combining} \eqref{eq:ExpectationIncrements1}-\eqref{eq:ExpectationIncrements3} \blue{with similar calculations as in the proof of Proposition \ref{Prop:CovariationComponentwise}} we get
\small{
\begin{align}
	\hat{\mathbb{E}} \left[ B_s^H(i) B_t^H(j) \right]	& \quad = \frac{1}{4} \hat{\mathbb{E}} \bigg[ \int_0^s K_H(s,u)K_H(t,u) d \langle B^{\textbf{e}_i+ \textbf{e}_j} \rangle_u \pm 2\int_0^s G_{\textbf{e}_i+ \textbf{e}_j} \left(K_H(s,u)K_H(t,u)\right) du \nonumber \\
	& \quad \quad + \int_0^s -K_H(s,u)K_H(t,u) d \langle B^{\textbf{e}_i- \textbf{e}_j} \rangle_u \pm 2\int_0^s G_{\textbf{e}_i- \textbf{e}_j} \left(-K_H(s,u)K_H(t,u)\right) du \bigg] \nonumber \\
	&\quad\leq \frac{1}{2} \left( \int_0^s G_{\textbf{e}_i+ \textbf{e}_j} \left(K_H(s,u)K_H(t,u)\right) du + \int_0^s G_{\textbf{e}_i- \textbf{e}_j} \left(-K_H(s,u)K_H(t,u)\right) du \right) \nonumber \\
	&\quad= \frac{1}{4} \bigg( \int_0^s  \sup_{\gamma \in \Theta} (\gamma_{ii}+ \gamma_{jj}+2\gamma_{ij}) \left(K_H(s,u)K_H(t,u)\right)^+  du  \notag \\ & \qquad - \int_0^s  \inf_{\gamma \in \Theta} (\gamma_{ii}+ \gamma_{jj}-2\gamma_{ij}) \left(-K_H(s,u)K_H(t,u)\right)^- du \bigg) \nonumber \\
	&\quad= \frac{1}{{4}} \bigg( \int_0^s  \sup_{\gamma \in \Theta} (\gamma_{ii}+ \gamma_{jj}+2\gamma_{ij}) K_H(s,u)K_H(t,u)  du \notag \\ & \qquad  - \int_0^s  \inf_{\gamma \in \Theta} (\gamma_{ii}+ \gamma_{jj}-2\gamma_{ij}) K_H(s,u)K_H(t,u) du \bigg) \label{eq:KHpositive} \\
	&\quad= \frac{1}{{8}} \left(t^{2H}+s^{2H}-\vert t-s \vert^{2H} \right) \left( \sup_{\gamma \in \Theta} (\gamma_{ii}+ \gamma_{jj}+2\gamma_{ij})-   \inf_{\gamma \in \Theta} (\gamma_{ii}+ \gamma_{jj}-2\gamma_{ij}) \right),\nonumber
\end{align}}
where \eqref{eq:KHpositive} follows since $K_H(t,s)>0$ for any $t>s>0$.
\allowdisplaybreaks
Analogously, we have
\begin{align}
	\hat{\mathbb{E}} \left[-B_s^H(i) B_t^H(j) \right]
	&= \hat{ \mathbb{E}} \left[- \int_0^s K_H(s,u) K_H(t,u) d \langle B(i),B(j) \rangle_u \right] \nonumber \\
	&= \frac{1}{4} \hat{\mathbb{E}} \left[ \int_0^s- K_H(s,u) K_H(t,u) d \langle B^{\textbf{e}_i+ \textbf{e}_j} \rangle_u + \int_0^s K_H(s,u) K_H(t,u)  d \langle B^{\textbf{e}_i -\textbf{e}_j} \rangle_u \right]\nonumber \\
	&= \frac{1}{4} \hat{\mathbb{E}} \bigg[ \int_0^s - K_H(s,u)K_H(t,u) d \langle B^{\textbf{e}_i+ \textbf{e}_j} \rangle_u \pm 2\int_0^s G_{\textbf{e}_i+ \textbf{e}_j} \left(-K_H(s,u)K_H(t,u)\right) du \nonumber \\
	& \quad \quad + \int_0^s K_H(s,u)K_H(t,u) d \langle B^{\textbf{e}_i- \textbf{e}_j} \rangle_u \pm 2\int_0^s G_{\textbf{e}_i- \textbf{e}_j} \left(K_H(s,u)K_H(t,u)\right) du \bigg] \nonumber \\
	&\leq \frac{1}{2} \left( \int_0^s G_{\textbf{e}_i+ \textbf{e}_j} \left(-K_H(s,u)K_H(t,u)\right) du + \int_0^s G_{\textbf{e}_i- \textbf{e}_j} \left(K_H(s,u)K_H(t,u)\right) du \right) \nonumber \\
	&= \frac{1}{{4}} \bigg( \int_0^s  \inf_{\gamma \in \Theta} (\gamma_{ii}+ \gamma_{jj}+2\gamma_{ij}) (-K_H(s,u)K_H(t,u) ) du  \notag \\ & \qquad + \int_0^s  \sup_{\gamma \in \Theta} (\gamma_{ii}+ \gamma_{jj}-2\gamma_{ij}) \left(K_H(s,u)K_H(t,u)\right) du \bigg) \nonumber \\
	&= \frac{1}{{8}} \left(t^{2H}+s^{2H}-\vert t-s \vert^{2H} \right) \left( \sup_{\gamma \in \Theta} (\gamma_{ii}+ \gamma_{jj}-2\gamma_{ij})-   \inf_{\gamma \in \Theta} (\gamma_{ii}+ \gamma_{jj}+2\gamma_{ij}) \right).\nonumber
\end{align}
\blue{Let now $H=\frac{1}{2}$. In this case, our fractional $G$-Brownian motion boils down to a $G$-Brownian motion. \greenNew{Consider} $0 \le s \le t$ and fix \greenNew{$i, j \in \lbrace 1,...,d \rbrace$,} $i \ne j$. We then have 
\begin{align}
\hat{\mathbb{E}} \left[B_s^H(i) B_t^H(j) \right] 
& = \hat{\mathbb{E}} \left[\hat{\mathbb{E}} \left[B_s^H(i) B_t^H(j) \bigg \vert\Omega_s\right]\right] \label{eq:H12fromprop323v} \\
& = \hat{\mathbb{E}} \left[\left(B_s^H(i)\right)^+\hat{\mathbb{E}} \left[B_t^H(j) \pm B_s^H(j)  \bigg \vert\Omega_s\right] + \left(B_s^H(i)\right)^-\hat{\mathbb{E}} \left[-B_t^H(j) \pm B_s^H(j)  \bigg \vert\Omega_s\right]\right] \label{eq:H12fromprop323iv}\\
& = \hat{\mathbb{E}} \left[B_s^H(j)\left(\left(B_s^H(i)\right)^+  - \left(B_s^H(i)\right)^-\right) \right] \label{eq:H12fromexamples324328}\\
& = \hat{\mathbb{E}} \left[B_s^H(j)B_s^H(i) \right] \nonumber\\
& =s \cdot \sup_{\gamma \in \Theta} \gamma_{ij} \label{eq:finalforH12},
\end{align}
where \eqref{eq:H12fromprop323v} and \eqref{eq:H12fromprop323iv} come from points (v) and (iv) of Proposition 3.2.3 in \cite{peng_nonlinearExpectation_book}, respectively, whereas \eqref{eq:H12fromexamples324328} is implied by Examples 3.2.4 and 3.2.8 in \cite{peng_nonlinearExpectation_book}. Moreover, \eqref{eq:finalforH12} follows by equation \eqref{eq:QuadraticVariationD_new}. \\Finally, equation \eqref{eq:secondforH12sitj} can be proved using similar arguments.
}
\end{proof}

We now prove that the components of a $d$-dimensional fractional $G$-Brownian motion {are self-similar and have stationary increments}.
	
	\begin{prop} \label{prop:SelfSimilar_and_StationarityIncrements}
	The components of a $d$-dimensional fractional $G$-Brownian motion \blue{of Hurst index $H \in \left(0,\frac{1}{2}\right) \cup \left(\frac{1}{2},1\right)$} are self-similar, i.e. for all $i =1,...,d$, $a >0$, it holds 
	$${a^{-H}B_{at}^H(i)\sim  B_t^H(i) \quad \text{for } t \geq 0.}$$
	For $H=\frac{1}{2}$ \blue{a} $d$-dimensional fractional $G$-Brownian motion is self-similar. 
	\end{prop}
	
	\begin{proof}
		Let $a >0$. By Remark 3.1.4 in \cite{peng_nonlinearExpectation_book} we know that $\left(a^{-\frac{1}{2}}B_{at}^{\frac{1}{2}}\right)_{t \geq 0}{=}\left(a^{-\frac{1}{2}}B_{at}\right)_{t \geq 0}$ is also a $d$-dimensional $G$-Brownian motion and thus $B^{\frac{1}{2}}$ is self-similar. Let now $H \in \left( 0, \frac{1}{2}\right) \cup \left( \frac{1}{2}, 1 \right)$. We first show that 
		\begin{align} \label{eq:KernelScalingProperty}
			K_H(at,u)=a^{H-\frac{1}{2}}K_H\left( t, \frac{u}{a}\right)
		\end{align}
		for $0<u<t$. For $H \in \left(\frac{1}{2},1\right)$, this follows since
		\begin{align*}
			K_H(at,u)&=c_H u^{\frac{1}{2}-H} \int_u^{at} (v-u)^{H-\frac{3}{2}} v^{H-\frac{1}{2}}dv \\
			&=c_H u^{\frac{1}{2}-H} \int_{\frac{u}{a}}^t(za-u)^{H-\frac{3}{2}} (za)^{H-\frac{1}{2}}a \ dz \\
			&=c_H u^{\frac{1}{2}-H} a^{2H-1}  \int_{\frac{u}{a}}^t(z-u)^{H-\frac{3}{2}} z^{H-\frac{1}{2}}  dz \\
			&=a^{H-\frac{1}{2}}c_H \left(\frac{u}{a}\right)^{\frac{1}{2}-H}  \int_{\frac{u}{a}}^t(z-u)^{H-\frac{3}{2}} z^{H-\frac{1}{2}} dz \\
			&=a^{H-\frac{1}{2}} K_H\left(t,\frac{u}{a}\right),
		\end{align*}
		where we perform a change of variables with $z:= \frac{v}{a}$. With similar arguments we derive \eqref{eq:KernelScalingProperty} also for $H \in \left( 0,\frac{1}{2}\right)$. By the definition of $B^H(i)$ and equation \eqref{eq:KernelScalingProperty} we get
		\begin{align*}
			a^{-H}B_{at}^H(i)=a^{-H}\int_0^{at} K_H(at,u)dB_u(i)=a^{-\frac{1}{2}}\int_0^{at} K_H\left( t, \frac{u}{a}\right) dB_u(i).
		\end{align*}
		Thus, by using that {$a^{-\frac{1}{2}}B_{at}(i)\sim B_t(i)$ for $t \geq 0$ }, it follows 
		\begin{equation*}
			{a^{-H}B_{at}^H(i)\sim  B_t^H(i) \quad \text{for } t \geq 0.}
		\end{equation*} 
		\end{proof}

The next lemma is used to prove both \blue{Propositions} \blue{\ref{prop:PDEFractionalBrownianMotion} and \ref{prop:statincr}}. 
\begin{lemma} \label{lemma:DistributionDeterministicFunction}
		Let $f:\mathbb{R}_+ \to \mathbb{R}$ such that 
		\begin{equation} \label{eq:IntegrabilityCondition}
			\int_0^{{T}} f^2(s)ds < \infty
		\end{equation} 
		{for all $T>0$.}
		Then for \blue{$i=1,...,d$}, $T>0$, it holds
		\begin{equation} \label{eq:DistributionDeterminsitic}
			\int_0^T f(t) dB_t(i) \sim \mathcal{N}\left( 0, [\underline{\sigma}_i^2(t), \overline{\sigma}_i^2(t)]\right)
	\end{equation}
 with $\underline{\sigma}^2_i(t):=\underline{\sigma}_i^2 \int_0^t f^2(s)ds$ and $\overline{\sigma}^2_i(t):=\overline{\sigma}_i^2 \int_0^t f^2(s)ds$.
 
 \end{lemma}
 \begin{proof}
Let $f: \mathbb{R}_+ \to \mathbb{R}$ such that $\eqref{eq:IntegrabilityCondition}$ holds. Then, for fixed $T>0$ it holds $f \in M_G^2(0,T)$ and the $G$-It\^{o} integral of $f$ with respect to $B(i)$ is well-defined.
 	We {first} prove \eqref{eq:DistributionDeterminsitic} for a {simple process} $f \in M_G^{2,0}(0,T)$. To do so, \blue{for any fixed $N \in \mathbb{N}$} let $\lbrace t_0,...,t_N \rbrace$ be a partition of $[0,T]$, and 
 \begin{equation*}
	f(t):=\sum_{j=0}^{N-1} f(j) \textbf{1}_{[t_j, t_{j+1}]}(t)
\end{equation*} 
with $f(j):=f(t_j)$.
Then by Definition 3.3.3 in \cite{peng_nonlinearExpectation_book} 
	\begin{equation} \label{eq:DefinitionIntegralDeterministic}
		I_{\blue{N}}(f)=\int_0^T f(t) dB_t(i):=\sum_{j=0}^{N-1} f(j)(B_{t_{j+1}}(i)-B_{ t_{j}}(i)).
	\end{equation} 
We now show via induction that for $I_{\blue{N}}(f)$ defined in \eqref{eq:DefinitionIntegralDeterministic} it holds
\begin{equation} \label{eq:DistributionSimpleFunction}
	I_{\blue{N}}(f) \sim \mathcal{N}\left( 0, \left[ \underline{\sigma}_i^2 a^N, \overline{\sigma}_i^2a^N\right]\right)
\end{equation}
with $a^N:=\sum_{j=0}^{N-1}(f(j))^2 (t_{j+1}-t_j).$ For $N=1$, we have $B_{t_1}(i)=B_T(i) \sim \mathcal{N}(0,[\underline{\sigma}_i^2 T, \overline{\sigma}_i^2T])$ and thus $f(T) B_T(i) \sim \mathcal{N}(0,[\underline{\sigma}_i^2 (f(T))^2T, \overline{\sigma}_i^2(f(T))^2T])$. We now assume that \blue{\eqref{eq:DistributionSimpleFunction}} holds for $\blue{N=n-1}.$ By the induction hypothesis we know that
\begin{equation} \label{eq:InductionStep}
	\sum_{j=0}^{n-2} f(j) (B_{t_{j+1}}(i)-B_{t_j}(i)) \sim  \mathcal{N}\left( 0, \left[ \underline{\sigma}_i^2 a^{\blue{n-1}}, \overline{\sigma}^2_ia^{\blue{n-1}}\right]\right)
\end{equation}
with $a^{\blue{n-1}}:=\sum_{j=0}^{n-2}(f(j))^2 (t_{j+1}-t_j)$. Moreover, it holds that 
$$
f(n-1) (B_{t_n}(i)-B_{t_{n-1}}(i))\sim  f({n-1})\sqrt{t_{n}-t_{n-1}}B_1(i).
$$ Note that by \eqref{eq:InductionStep} \blue{we have}
$$
\sum_{j=0}^{n-2} f(j) (B_{t_{j+1}}(i)-B_{t_j}(i))\sim  \sqrt{a^{\blue{n-1}}}\bar{B}_1(i),
$$
 where \blue{$\bar{B}(i)$ is a $G$-Brownian motion and} $\bar{B}_1(i)$ is independent of $B_1(i)$ as $B_{t_n}(i)-B_{t_{n-1}}(i)$ is independent of the vector $(B_{t_1}(i),B_{t_2}(i)-B_{t_1}(i),...,B_{t_{n-1}}(i)-B_{t_{n-2}}(i))$. Therefore, by Definition \ref{defi:G-NormalDistribution} 	
\begin{align*}
	&f({n-1})\sqrt{t_{n}-t_{n-1}}B_1(i) + a^{\blue{n-1}} \bar{B}_1(i) \nonumber \\ &\sim  \sqrt{(f({n-1}))^2(t_n-t_{n-1})+ \sum_{j=0}^{n-2}(f(j))^2 (t_{j+1}-t_j)} B_1(i) \\&= \sqrt{\blue{a^{n}}} B_1(i),
\end{align*}

which concludes the proof of the induction. \\
From now on we denote by $\pi_T^N=\lbrace t_0^N,t_1^N,...,t_N^N \rbrace$\blue{, $N \in \mathbb{N}$,} a sequence of partitions of $[0,T]$ such that $\lim_{N \to \infty} \mu(\pi_T^N)=0$ \blue{with $\mu(\pi_T^N)$ defined in \eqref{eq:MeshSize}.}
By Lemma 3.3.4 in \cite{peng_nonlinearExpectation_book} the mapping $I: M_G^{0,2}(0,T) \to L_G^2(\Omega_T)$ in \eqref{eq:DefinitionIntegralDeterministic} can be continuously extended to $I: M_G^2(0,T) \to L_G^2(\Omega_T)$. Moreover, by the definition of the Lebesgue integral it is clear that 
\begin{align*}
	\lim_{N \to \infty} a^N=\lim_{N \to \infty}\sum_{j=0}^{N-1} (f_j^N)^2(t_{j+1}^N-t_{j}^N)= \int_0^T f^2(s) ds
\end{align*}
with $f_j^N:=f(t_j^N).$ Thus, combining all the results the lemma follows.
 \end{proof}

 {We close the section with \blue{the} following results}.

\begin{lemma} \label{lemma:DistributionFractionalGBM}
Let $B^H=(B_t^H)_{t \geq 0}$ be a $d$-dimensional fractional $G$-Brownian motion \blue{of} Hurst \blue{index} $H \in (0,1)$. Then for $i=1,...,d$ and fixed $t \geq 0$
\begin{equation*}
	B_t^H(i)  \sim \mathcal{N}\left(0, [\underline{\sigma}_i^2  t^{2H}, \overline{\sigma}_i^2 t^{2H}] \right). 
\end{equation*}
In particular for $t=1$ and any $H \in (0,1)$ it holds $B_1^H(i) {\sim}B_1(i)$.
\end{lemma}

\begin{proof}
	Let $H \in (0,1)$ and fix $t \geq 0$. Then by \eqref{eq:IntegralKernels} we have
 \begin{equation*}
 \int_0^t (K_H(t,u))^2 du = \frac{1}{2}\left( t^{2H} + t^{2H}\right)	= t^{2H}< \infty
 \end{equation*}
 and thus it follows by Lemma \ref{lemma:DistributionDeterministicFunction} that $B_t^H(i) \sim \mathcal{N}\left(0, [\underline{\sigma}_i^2  t^{2H}, \overline{\sigma}^2_i t^{2H}] \right)$. 
\end{proof}

{The next proposition provides a useful connection between \blue{the} G-expectation \blue{of a \greenNew{multi-dimensional} fractional $G$-Brownian motion} and nonlinear PDEs.}

\begin{prop} \label{prop:PDEFractionalBrownianMotion}
	{Let $B^H=(B_t^H)_{t \geq 0}$ be a $d$-dimensional fractional $G$-Brownian motion \blue{of} Hurst \blue{index} $H \in (0,1)$} and fix $i=1,...,d$, $t \geq 0$. 
	Then 
	for each $\varphi \in C_{\blue{l,lip}}(\mathbb{R})$ the function
	\begin{equation*}
			\overline{u}_i(t,x):=\hat{\mathbb{E}}[\varphi(x+t^{-H+\frac{1}{2}}B_t^H(i))], \quad (t,x) \in [0, \infty) \times \mathbb{R}
	\end{equation*}
	is the unique viscosity solution of the following PDE
	\begin{equation} \label{eq:StandardGHeatEquation}
	\partial_t u -G_i(D_x^2u)=0, \quad u\vert_{t=0}= \varphi,
	\end{equation}
where $G_i: \mathbb{R} \to \mathbb{R}$ is defined in \eqref{eq:DefinitionGComponent}.
	\end{prop}
	\begin{proof}
	Fix $t \ge 0$ and $i=1,...,d$. Let $H \in \left(0, \frac{1}{2} \right)\cup \left(\frac{1}{2},1 \right)$. By Lemma \ref{lemma:DistributionFractionalGBM} we get 
 	\begin{equation} \label{eq:RewriteValueFunction}
		\bar{u}_{\blue{i}}(t,x):=\hat{\mathbb{E}}[\varphi(x+t^{-H+\frac{1}{2}}B_t^H(i))]=\hat{\mathbb{E}}[\varphi(x+t^{H}t^{-H+\frac{1}{2}}B_1^H(i))]=\hat{\mathbb{E}}[\varphi(x+\sqrt{t}B_1(i))],
	\end{equation}	
	where the second equality follows since {$B_t^H(i) \sim t^H B_1^H(i)$ for $t \geq 0$} by Proposition \ref{prop:SelfSimilar_and_StationarityIncrements}.
	
	Therefore, if follows by Example 2.2.14 in \cite{peng_nonlinearExpectation_book} that \blue{$\bar{u}_i$} is the unique viscosity solution of the standard $G_i$-heat equation in \eqref{eq:StandardGHeatEquation}. Note that for $H= \frac{1}{2}$ the result follows by Theorem 3.1.3 in \cite{peng_nonlinearExpectation_book}.
	\end{proof}

\subsection{Increments of \blue{a} multi-dimensional {$G$-fBm}}
We now investigate some properties of the increments of the components of a $d$-dimensional fractional $G$-Brownian motion. {First of all we show the stationarity of the increments of each component.}

\begin{prop}\label{prop:statincr}
{Let $B^H=(B_t^H)_{t \geq 0}$ be a $d$-dimensional fractional $G$-Brownian motion \blue{of} Hurst \blue{index} $H \in (0,1)$.} For $H \in \ (0,\frac{1}{2}) \cup (\frac{1}{2},1)$ {and} $i=1,...,d$ 
$${B_{t+h}^H(i)-B_h^H(i)\sim B_t^H(i) \quad \text{for }t \geq 0}, \quad \blue{h>0}.$$
{For $H =\frac{1}{2}$, the process $B^H$ is a $G$-Brownian motion and has therefore stationary increments.}
\end{prop}
\begin{proof}
First, note that a one-dimensional fractional Brownian motion $(\bar{B}^H_t)_{t \geq 0}$ on a standard probability space {$(\Omega, \mathcal{F}, {P})$} has stationary increments. In particular, $\bar{B}_{t+h}^H-\bar{B}_t^H$ has the same distribution as $\bar{B}_t^H$ for every $t,h>0$ {under $P$}. Let now $H \in (\frac{1}{2},1)$. Then, by the Volterra representation of a fractional Brownian motion and since $K_H(t,s)=0$ for \blue{$0 \leq t \leq s$}, we have
	\begin{align}
		\bar{B}_{t+h}^H-\bar{B}_h^H&=\int_{0}^{t+h} K_H(t+h,s) d\bar{B}_s - \int_{0}^{h} K_H(h,s) d\bar{B}_s \nonumber \\
		&=\int_0^{t+h} \left( K_H(t+h,s)-K_H(h,s) \right)d\bar{B}_s \label{eq:StationarityI},
	\end{align}
	 where $(\bar{B}_t)_{t \geq 0}$ denotes a standard Brownian motion. Therefore, applying Lemma \ref{lemma:DistributionDeterministicFunction} in the classical case when $\overline{\sigma}_i^2=\underline{\sigma}^2_i=1$, we get
	$$
	\bar{B}_{t+h}^H-\bar{B}_h^H \sim \bar{\mathcal{N}} \left( 0,  g(t+h)\right) \quad \text{ and } \quad \bar{B}_t^H \sim \bar{\mathcal{N}}(0,h(t))
	$$
	with $g(t+h):=\int_0^{t+h}  \left( K_H(t+h,s)-K_H(h,s) \right)^2 ds$, $h(t):=\int_0^t \blue{(K_H(t,s))^2} ds$ \blue{and where $\bar{\mathcal{N}}$ denotes the standard normal distribution}. Thus, by the stationarity of the increments of a fractional Brownian motion, it follows that 
	\begin{equation} \label{eq:StationarityII}
		g(t+h)=h(t).
	\end{equation}
 For the components of a $d$-dimensional fractional $G$-Brownian motion, by similar considerations as in \eqref{eq:StationarityI} and by Lemma \ref{lemma:DistributionDeterministicFunction} we get for $i=1,...,d$
	$$
	B_{t+h}^H(i)-B_h^H(i) \sim \mathcal{N} \left( 0, \left[ \underline{\sigma}_i^2 g(t+h), \overline{\sigma}_i^2 g(t+h)  \right]\right) 
	$$
	and 
	$$
	B_t^H(i) \sim \mathcal{N} \left( 0, \left[ \underline{\sigma}_i^2 h(t), \overline{\sigma}_i^2 h(t) \right]\right).
	$$
	Thus, \eqref{eq:StationarityII} implies that the components of {a} $d$-dimensional fractional $G$-Brownian motion have stationary increments. Similar arguments can be used for $H \in (0,\frac{1}{2})$. Moreover, for $H=\frac{1}{2}$ the result follows directly by the stationarity of the increments of \blue{a} $d$-dimensional $G$-Brownian motion.
\end{proof}

{
\begin{lemma}\label{lemma:condexpectationintegrals}
Fix $i, j=1,...,d$ with $i\neq j$ and $0 \le w \le s < r < u < t$. \blue{Let $B^H=(B_t^H)_{t \geq 0}$ be a $d$-dimensional fractional $G$-Brownian motion of Hurst index} $H \in (0,\frac{1}{2}) \cup (\frac{1}{2},1)$. Then it holds 
	\small{\begin{align}
	&\hat{\mathbb{E}}\left[\left(\int_w^r K_H(r,v)dB_v(j) -\int_w^s K_H(s,v)dB_v(j)\right)\left(\int_w^t K_H(t,v)dB_v(i) -\int_w^u K_H(u,v)dB_v(i)\right) \bigg| \Omega_w \right] \nonumber \\
	&\leq \frac{1}{8} \overline \sigma_{ij}^2\left((t-s)^{2H}-(u-s)^{2H}-(t-r)^{2H}+(u-r)^{2H} \right)\nonumber
	\intertext{and}
	&\hat{\mathbb{E}}\left[\left(\int_w^r K_H(r,v)dB_v({i}) -\int_w^s K_H(s,v)dB_v({i})\right)\left(\int_w^t K_H(t,v)dB_v(i) -\int_w^u K_H(u,v)dB_v(i)\right) \bigg| \Omega_w \right]\nonumber \\
	&=\overline \sigma_i^2 \left((t-s)^{2H}-(u-s)^{2H}-(t-r)^{2H}+(u-r)^{2H} \right) \nonumber \\ 
	& \quad -\overline{\sigma}_i^2 \int_0^w \left(K_H(r,v)- K_H(s,v)\right)\left(K_H(t,v)-K_H(u,v)\right) dv.\nonumber
	\end{align}}
	{Furthermore, it holds
	\small{
	\begin{align}
	&-\hat{\mathbb{E}}\left[-\left(\int_w^r K_H(r,v)dB_v(j) -\int_w^s K_H(s,v)dB_v(j)\right)\left(\int_w^t K_H(t,v)dB_v(i) -\int_w^u K_H(u,v)dB_v(i)\right) \bigg| \Omega_w \right] \nonumber \\
	&\geq \frac{1}{8} \underline \sigma_{ij}^2\left((t-s)^{2H}-(u-s)^{2H}-(t-r)^{2H}+(u-r)^{2H} \right)\nonumber
	\intertext{and}
	&-\hat{\mathbb{E}}\left[-\left(\int_w^r K_H(r,v)dB_v({i}) -\int_w^s K_H(s,v)dB_v({i})\right)\left(\int_w^t K_H(t,v)dB_v(i) -\int_w^u K_H(u,v)dB_v(i)\right) \bigg| \Omega_w \right]\nonumber \\
	&=\underline \sigma_i^2 \left((t-s)^{2H}-(u-s)^{2H}-(t-r)^{2H}+(u-r)^{2H} \right) \nonumber \\ 
	& \quad -\underline{\sigma}_i^2 \int_0^w \left(K_H(r,v)- K_H(s,v)\right)\left(K_H(t,v)-K_H(u,v)\right) dv. \nonumber
	\end{align}}}
\end{lemma}}

\begin{proof}
{Consider first some general $i, j=1,...,d$. We have 
\allowdisplaybreaks
\small{
\begin{align}
& \hat{\mathbb{E}}\left[ \left(\int_w^r K_H(r,v)dB_v(j)-\int_w^s K_H(s,v)dB_v(j)\right)   \left(\int_w^t K_H(t,v)dB_v(i)-\int_w^u K_H(u,v)dB_v(i)\right) \bigg| \Omega_w \right] \nonumber \\
& = \hat{\mathbb{E}}\bigg[\hat{\mathbb{E}}\bigg[ \left(\int_w^r K_H(r,v)dB_v(j)-\int_w^s K_H(s,v)dB_v(j)\right)   \left(\int_w^r K_H(t,v)dB_v(i)-\int_w^r K_H(u,v)dB_v(i)\right)  \nonumber \\
&\quad + \left(\int_w^r K_H(r,v)dB_v(j)-\int_w^s K_H(s,v)dB_v(j)\right)   \left(\int_r^t K_H(t,v)dB_v(i)-\int_r^u K_H(u,v)dB_v(i)\right)\bigg| \Omega_r \bigg] \bigg| \Omega_w \bigg] \nonumber \\ 
& = \hat{\mathbb{E}}\bigg[\left(\int_w^r K_H(r,v)dB_v(j)-\int_w^s K_H(s,v)dB_v(j)\right)   \left(\int_w^r K_H(t,v)dB_v(i)-\int_w^r K_H(u,v)dB_v(i)\right)   \nonumber \\
&\quad + \hat{\mathbb{E}}\bigg[ \left(\int_w^r K_H(r,v)dB_v(j)-\int_w^s K_H(s,v)dB_v(j)\right)   \left(\int_r^t K_H(t,v)dB_v(i)-\int_r^u K_H(u,v)dB_v(i)\right)\bigg| \Omega_r \bigg] \bigg| \Omega_w\bigg]\label{eq:becausekerneldeterministic} \\ 
& = \hat{\mathbb{E}}\left[\left(\int_w^r K_H(r,v)dB_v(j)-\int_w^s K_H(s,v)dB_v(j)\right)  \int_w^r (K_H(t,v)-K_H(u,v))dB_v(i) \bigg| \Omega_w  \right] \label{eq:becausecondexpzero} \\
& = \hat{\mathbb{E}}\left[\left(\int_w^r K_H(r,v)dB_v(j)-\int_w^r K_H(s,v)dB_v(j)\right)  \int_w^r (K_H(t,v)-K_H(u,v))dB_v(i) \bigg| \Omega_w \right] \label{eq:becauseintegrandzero} \\
& = \hat{\mathbb{E}}\left[\int_w^r \left(K_H(r,v)- K_H(s,v)\right)dB_v(j)   \int_w^r \left(K_H(t,v)-K_H(u,v)\right)dB_v(i) \bigg| \Omega_w \right] \label{eq:forgeneraliandj},
\end{align}}
where \eqref{eq:becausekerneldeterministic} follows \blue{by Remark 3.2.4 in \cite{peng_nonlinearExpectation_book}} whereas \eqref{eq:becauseintegrandzero} comes from the fact that $K_H(s,v)=0$ for any $0 \le s \le v$, so that 
$$
\int_s^r K_H(s,v)dB_v(j)=0.
$$
Moreover, \eqref{eq:becausecondexpzero} follows because from Proposition 3.2.3(iv) in \cite{peng_nonlinearExpectation_book} we get
  \begin{align}
	&\hat{\mathbb{E}}\bigg[ \left(\int_w^r K_H(r,v)dB_v(j)-\int_w^s K_H(s,v)dB_v(j)\right)   \left(\int_r^t K_H(t,v)dB_v(i)-\int_r^u K_H(u,v)dB_v(i)\right)\bigg| \Omega_r \bigg] \nonumber \\ 
	&= \left(\int_w^r K_H(r,v)dB_v(j)-\int_w^s K_H(s,v)dB_v(j)\right)^+\hat{\mathbb{E}}\bigg[\int_r^t K_H(t,v)dB_v(i)-\int_r^u K_H(u,v)dB_v(i)\bigg|\Omega_{{r}}\bigg] \nonumber \\ & \quad + \left(\int_w^r K_H(r,v)dB_v(j)-\int_w^s K_H(s,v)dB_v(j)\right)^-\hat{\mathbb{E}}\bigg[-\int_r^t K_H(t,v)dB_v(i)+\int_r^u K_H(u,v)dB_v(i)\bigg|\Omega_{{r}}\bigg] \nonumber \\ & = 0 \nonumber,
 \end{align}
Suppose now $i=j$ {and set
 \begin{equation} \nonumber
 F(v):=F^{r,s,t,u}(v):=\left(K_H(r,v)- K_H(s,v)\right)\left(K_H(t,v)-K_H(u,v)\right), \quad v \ge 0.
 \end{equation}
}By \eqref{eq:forgeneraliandj} and the $G$-It\^{o}'s formula applied to the process 
 $$
(Y^{r,s,t,u,w}_z(i))_{z \geq w}=\left(\int_w^z \left(K_H(r,v)- K_H(s,v)\right)dB_v(j)   \int_w^z \left(K_H(t,v)-K_H(u,v)\right)dB_v(i)\right)_{z \ge w},
 $$
 it holds
\begin{align}
	&  \hat{\mathbb{E}}\left[ \left(\int_w^r K_H(r,v)dB_v(i)-\int_w^s K_H(s,v)dB_v(i)\right) \left(\int_w^t K_H(t,v)dB_v(i)-\int_w^u K_H(u,v)dB_v(i)\right) \bigg| \Omega_w \right]\nonumber \\ 
	&=\hat{\mathbb{E}}\left[  \int_w^r F(v)d\langle B(i) \rangle_v - 2\int_w^r G_i \left( F(v)\right)dv  \bigg| \Omega_w \right] + 2\int_w^r G_i \left( F(v)\right)dv \nonumber \\
	&= 2\int_w^r G_i \left( F(v)\right)dv \label{eq:noexpanymore} \\
	&= \overline{\sigma}_i^2 \int_w^r \left(K_H(r,v)- K_H(s,v)\right)\left(K_H(t,v)-K_H(u,v)\right) dv  \label{eq:onlyoverlinesigma} \\
	&= \overline{\sigma}_i^2 \int_0^r \left(K_H(r,v)- K_H(s,v)\right)\left(K_H(t,v)-K_H(u,v)\right) dv \nonumber \\ 
	& \quad- \overline{\sigma}_i^2 \int_0^w \left(K_H(r,v)- K_H(s,v)\right)\left(K_H(t,v)-K_H(u,v)\right) dv \nonumber \\
	&= \overline{\sigma}_i^2 \int_0^r K_H(r,v))\left(K_H(t,v)-K_H(u,v)\right) dv - \overline{\sigma}_i^2 \int_0^s K_H(s,v)\left(K_H(t,v)-K_H(u,v)\right) dv \nonumber \\ 
	& \qquad- \overline{\sigma}_i^2 \int_0^w \left(K_H(r,v)- K_H(s,v)\right)\left(K_H(t,v)-K_H(u,v)\right) dv \nonumber \\
	&=\overline{\sigma}_{i}^2 \left(t^{2H}-u^{2H}-(t-r)^{2H}+(u-r)^{2H} - \left(t^{2H}-u^{2H}-(t-s)^{2H}+(u-s)^{2H} \right) \right)\notag \\
	& \quad- \overline{\sigma}_i^2 \int_0^w \left(K_H(r,v)- K_H(s,v)\right)\left(K_H(t,v)-K_H(u,v)\right) dv \label{eq:finalfromincrements} \\
 	&=\overline \sigma_i^2 \left((t-s)^{2H}-(u-s)^{2H}-(t-r)^{2H}+(u-r)^{2H} \right) \nonumber \\ 
	& \quad -\overline{\sigma}_i^2 \int_0^w \left(K_H(r,v)- K_H(s,v)\right)\left(K_H(t,v)-K_H(u,v)\right) dv. \nonumber
\end{align}
	Here \eqref{eq:noexpanymore} holds since $\left(\int_0^s F(v) d \langle B(i)\rangle_v - 2 \int_0^s G_i(F(v))dv\right)_{s \ge 0}$ is a $G_i$-martingale {by Proposition 4.1.4 in  \cite{peng_nonlinearExpectation_book}}, and we get \eqref{eq:onlyoverlinesigma} since 
$$
 \left(K_H(r,v)- K_H(s,v)\right)\left(K_H(t,v)-K_H(u,v)\right)>0.
 $$ 
 Finally, \eqref{eq:finalfromincrements} follows directly from \eqref{eq:IntegralKernels}.
}	
	
\small{Suppose now $i \ne j$.  In this case, using similar arguments as above, we have
\allowdisplaybreaks
\begin{align}
	&  \hat{\mathbb{E}}\left[ \left(\int_w^r K_H(r,v)dB_v(j)-\int_w^s K_H(s,v)dB_v(j)\right)   \left(\int_w^t K_H(t,v)dB_v(i)-\int_w^u K_H(u,v)dB_v(i)\right) \bigg| \Omega_w \right]\nonumber \\ 
	&= \frac{1}{4} \hat{\mathbb{E}}\bigg[ \int_w^r  \left(K_H(r,v)- K_H(s,v)\right)\left( K_H(t,v) - K_H(u,v)\right) d\langle B(j)+B(i)\rangle_v \notag \\ 
	& \qquad \qquad \pm 2 \int_w^r G_{\textbf{e}_i+\textbf{e}_j}\left( \left(K_H(r,v)- K_H(s,v)\right)\left( K_H(t,v) - K_H(u,v)\right)\right)dv \nonumber \\
	&\quad \qquad +  \int_w^r - \left(K_H(r,v)- K_H(s,v)\right)\left( K_H(t,v) - K_H(u,v)\right) d\langle B(j)-B(i)\rangle_v \notag \\
	& \qquad \qquad \pm 2 \int_w^r G_{\textbf{e}_i-\textbf{e}_j}\left(- \left(K_H(r,v)- K_H(s,v)\right)\left( K_H(t,v) - K_H(u,v)\right)\right)dv \bigg| \Omega_w \bigg] \nonumber \\
	&\leq \frac{1}{2} \bigg( \int_w^r G_{\textbf{e}_i+\textbf{e}_j}\left( \left(K_H(r,v)- K_H(s,v)\right)\left( K_H(t,v) - K_H(u,v)\right)\right)dv \notag \\
	&\qquad \quad + \int_w^r G_{\textbf{e}_i-\textbf{e}_j}\left(- \left(K_H(r,v)- K_H(s,v)\right)\left( K_H(t,v) - K_H(u,v)\right)\right)dv  \bigg) \nonumber \\
	&= \frac{1}{4} \bigg[ \int_w^r \sup_{\gamma \in \Theta}(\gamma_{ii}+ \gamma_{jj}+ 2 \gamma_{ij}) \left(K_H(r,v)- K_H(s,v)\right)\left( K_H(t,v) - K_H(u,v)\right)dv \nonumber \\
	&\quad - \int_w^r \inf_{\gamma \in \Theta }(\gamma_{ii} + \gamma_{jj}-2 \gamma_{ij}) \left(K_H(r,v)- K_H(s,v)\right)\left( K_H(t,v) - K_H(u,v)\right)dv  \bigg]  \nonumber \\
	& = \frac{1}{8} \overline{\sigma}_{ij}^2  \int_w^r \left(K_H(r,v)-K_H(s,v)\right)\left( K_H(t,v) - K_H(u,v)\right)dv  \nonumber \\
	& \le \frac{1}{8} \overline{\sigma}_{ij}^2  \int_0^r \left(K_H(r,v)-K_H(s,v)\right)\left( K_H(t,v) - K_H(u,v)\right)dv  \nonumber \\
	& = \frac{1}{8} \overline{\sigma}_{ij}^2 \left( \int_0^rK_H(r,v)\left( K_H(t,v) - K_H(u,v)\right)dv -  \int_0^sK_H(s,v)\left( K_H(t,v) - K_H(u,v)\right)dv \right) \nonumber \\
	&=\frac{1}{8} \overline{\sigma}_{ij}^2 \left(t^{2H}-u^{2H}-(t-r)^{2H}+(u-r)^{2H} - \left(t^{2H}-u^{2H}-(t-s)^{2H}+(u-s)^{2H} \right) \right)\notag \\
	&=\frac{1}{8} \overline{\sigma}_{ij}^2 \left((t-s)^{2H}-(u-s)^{2H}-(t-r)^{2H}+(u-r)^{2H} \right)\notag.
\end{align}}

With similar arguments the two remaining results of the lemma follow.
\end{proof}

The next results regard the correlation between two increments of \blue{a} fractional $G$-Brownian motion. Their proofs \blue{for $H \in (0,\frac{1}{2}) \cup (\frac{1}{2},1)$} follow directly from Lemma \ref{lemma:condexpectationintegrals}. \blue{When $H=\frac{1}{2}$, the results can be proved using points (v) and (iv) of Proposition 3.2.3 in \cite{peng_nonlinearExpectation_book} together with Example 3.2.8 in \cite{peng_nonlinearExpectation_book}, in a similar way as in the proof of Proposition \ref{prop:sitj} for $H=\frac{1}{2}$.}
	
	\begin{cor}\label{prop:fracGincrements}
	Let $B^H=(B^H_t)_{t \geq 0}$ be a $d$-dimensional fractional $G$-Brownian motion of Hurst index $H \in \left( 0,1\right)$. Fix $0 \le s < r < u  < t$ \blue{and $i=1,...,d$}. Then for any $H \in (0,\frac{1}{2}) \cup (\frac{1}{2},1)$ it holds  
	\small{\begin{align}
	\hat{\mathbb{E}}\left[\left(B_r^H(i)-B_s^H(i)\right)\left(B_t^H(i)-B_u^H(i)\right)\right]&=\overline \sigma_i^2\left((t-s)^{2H}-(u-s)^{2H}-(t-r)^{2H}+(u-r)^{2H} \right)\label{eq:incrementsfirst}
	\intertext{and}
	-\hat{\mathbb{E}}\left[-\left(B_r^H(i)-B_s^H(i)\right)\left(B_t^H(i)-B_u^H(i)\right)\right]&=\underline \sigma_i^2\left((t-s)^{2H}-(u-s)^{2H}-(t-r)^{2H}+(u-r)^{2H} \right)\label{eq:incrementssecond}.
	\end{align}}
	For $H=\frac{1}{2}$, \blue{we have}
	\small{\begin{align}
	\hat{\mathbb{E}}\left[\left(B_r^{\frac{1}{2}}\blue{(i)}-B_s^{\frac{1}{2}}\blue{(i)}\right)\left(B_t^{\frac{1}{2}}\blue{(i)}-B_u^{\frac{1}{2}}\blue{(i)}\right)\right]= -\hat{\mathbb{E}}\left[-\left(B_r^{\frac{1}{2}}\blue{(i)}-B_s^{\frac{1}{2}}\blue{(i)}\right)\left(B_t^{\frac{1}{2}}\blue{(i)}-B_u^{\frac{1}{2}}\blue{(i)}\right)\right]=0.
	\end{align}}
	\end{cor}

\begin{cor}
Let $B^H=(B^H_t)_{t \geq 0}$ be a $d$-dimensional fractional $G$-Brownian motion of Hurst index $H \in \left( 0,1\right)$. Fix $i, j=1,...,d$ with $i\neq j$ and $0 \le s < r < u < t$. Then for any $H \in (0,\frac{1}{2}) \cup (\frac{1}{2},1)$  it holds 
	\begin{align}
	&\hat{\mathbb{E}}\left[\left(B_r^H(j)-B_s^H(j)\right)\left(B_t^H(i)-B_u^H(i)\right)\right]\leq \frac{1}{8} \overline \sigma_{ij}^2\left((t-s)^{2H}-(u-s)^{2H}-(t-r)^{2H}+(u-r)^{2H} \right)\nonumber
	\intertext{and}
	&-\hat{\mathbb{E}}\left[-\left(B_r^H(j)-B_s^H(j)\right)\left(B_t^H(i)-B_u^H(i)\right)\right] \geq \frac{1}{8} \underline \sigma_{ij}^2\left((t-s)^{2H}-(u-s)^{2H}-(t-r)^{2H}+(u-r)^{2H} \right)\nonumber.
	\end{align}
		\blue{For $H=\frac{1}{2}$, we have
	\begin{align}
	\hat{\mathbb{E}}\left[\left(B_r^{\frac{1}{2}}(j)-B_s^{\frac{1}{2}}(j)\right)\left(B_t^{\frac{1}{2}}(i)-B_u^{\frac{1}{2}}(i)\right)\right] = -\hat{\mathbb{E}}\left[-\left(B_r^{\frac{1}{2}}(j)-B_s^{\frac{1}{2}}(j)\right)\left(B_t^{\frac{1}{2}}(i)-B_u^{\frac{1}{2}}(i)\right)\right]=0.\nonumber
	\end{align}
	}
\end{cor}

	Using \eqref{eq:incrementsfirst} and \eqref{eq:incrementssecond}, we now show that the components of a $d$-dimensional fractional $G$-Brownian motion \blue{exhibit} long-range dependence as the classical fractional Brownian motion for \blue{$H \in (\frac{1}{2},1)$,} see Section 1.4 in \cite{biagini_hu_oksendal_zhang_2008}. To do so, we introduce \blue{an} autocovariance function in the $G$-expectation setting.

	\begin{defi} \label{defi:Autocovariance}
		Let $X=(X_n)_{n \in \mathbb{N}}$ be a stationary process. Then we define the \emph{autocovariance function} ${\rho}:=(\overline{\rho}, \underline{\rho})$ by
		\begin{align}
		\overline{\rho}(n)&:= \hat{\mathbb{E}}\left[ \left( X_k-\hat{\mathbb{E}}[X_k]\right) \left( X_{k+n}-\hat{\mathbb{E}}[X_{k+n}]\right)\right], \label{eq:AutoCovariance1}\\
		\underline{\rho}(n)&:=- \hat{\mathbb{E}}\left[- \left( X_k-\hat{\mathbb{E}}[X_k]\right) \left( X_{k+n}-\hat{\mathbb{E}}[X_{k+n}]\right)\right]. \label{eq:AutoCovariance2}
		\end{align}
\end{defi}
	We now generalize the definition of a long-memory process introduced in \cite{Beran_Terrin_1996}. 
\begin{defi} \label{defi:LongMemory}
		Let $X=(X_n)_{n \in \mathbb{N}}$ be a stationary process. Then $X$ satisfies the property of \emph{long memory} if the autocovariance function $\rho=(\overline{\rho}, \underline{\rho})$ in Definition \ref{defi:Autocovariance} satisfies  
	\begin{align*}
		\sum_{n=1}^{\infty} \underline{\rho}(n)= \infty, \quad \sum_{n=1}^{\infty} \overline{\rho}(n)= \infty.
	\end{align*}
	If on the contrary
	\begin{align*}
		\sum_{n=1}^{\infty} \underline{\rho}(n)< \infty, \quad \sum_{n=1}^{\infty} \overline{\rho}(n)< \infty.
	\end{align*}
	we say that $X$ has a \emph{short memory}.
\end{defi}
The following proposition is a direct implication of   \eqref{eq:incrementsfirst} and \eqref{eq:incrementssecond}.	
	\begin{prop}
	Let $B^H=(B^H_t)_{t \geq 0}$ be a $d$-dimensional fractional $G$-Brownian motion of Hurst index $H \in \left( 0,1\right)$. {For any $i=1, \dots, d$, the increments process $X(i)=(X_k(i))_{k\in \mathbb{N}}$ with $X_{k}(i):=B^H_k(i)-B^H_{k-1}(i)$, $k\in \mathbb{N}$, has} long memory for \blue{$H \in (\frac{1}{2},1)$}, and short memory for \blue{$H \in (0,\frac{1}{2}]$}. 
	 \end{prop}
	 \begin{proof}
	 Note that the increments of \blue{a} fractional $G$-Brownian motion \blue{are} stationary by Proposition \blue{\ref{prop:statincr}}. Furthermore, we have
	\begin{align}
		\overline{\rho}(n)&=\hat{\mathbb{E}}\left[ \left( X_k(i)-\hat{\mathbb{E}}[X_k(i)]\right) \left( X_{k+n}(i)-\hat{\mathbb{E}}[X_{k+n}(i)]\right)\right] \nonumber \\
		&=\hat{\mathbb{E}}\left[ \left(B^H_k(i)-B^H_{k-1}(i)\right) \left(B^H_{k+n}(i)-B^H_{k+n-1}(i)\right)\right] \nonumber \\
		&=\hat{\mathbb{E}}\left[ B_1^H(i) \left(B^H_{n+1}(i)-B^H_{n}(i)\right)\right] \nonumber \\
		&= \frac{1}{2} \overline{\sigma}_i^2 \left( (n+1)^{2H}+(n-1)^{2H}-2n^{2H} \right), \nonumber
	\end{align}
	where the last equality follows directly from \eqref{eq:incrementsfirst}. In the same way, we get
	\begin{align}
		\underline{\rho}(n)= \frac{1}{2} \underline{\sigma}_i^2 \left( (n+1)^{2H}+(n-1)^{2H}-2n^{2H} \right)\blue{.} \nonumber
	\end{align}
	The result then follows since for large $n$ we have
	\begin{equation} \nonumber
		\frac{1}{2}\left( (n+1)^{2H}+(n-1)^{2H}-2n^{2H} \right) \sim H(2H-1) n^{2H-2}.
	\end{equation}
	\end{proof}
{We now show that our approach coincides with the one of \cite{fBrownianMotionGSetting1} in the one-dimensional case.}
{\begin{lemma} \label{lemma:DefinitionCorrespondsD1}
Let $B^H=(B^H_t)_{t \geq 0}$ be a one-dimensional fractional $G$-Brownian motion of Hurst index $H \in \left( 0,1\right)$ given in Definition \ref{def:FractionalGBrownianMotionOurDDimension}. Then, any continuous modification of $B^H$ is a one-dimensional fractional $G$-Brownian motion in the sense of Definition \ref{def:fBrownianMotionSetting1}. 
 \end{lemma}
 \begin{proof}
 By Proposition \ref{Prop:CovariationComponentwise} and Lemma \ref{lemma:DistributionFractionalGBM} it follows that $B^H$ satisfies the Points 1. and 2. in Definition \ref{def:fBrownianMotionSetting1}. Moreover, by Proposition \ref{prop:statincr} Point 3. follows. Finally, for $m \in \mathbb{N}$ and $m \geq 3$ we have
		\begin{align}
			\hat{\mathbb{E}} \left[\vert B^{H}_t\vert^m\right]t^{-2H}&=\hat{\mathbb{E}} \left[\vert B^{H}_1 t^{H}\vert^m\right]t^{-2H}= \hat{\mathbb{E}} \left[\vert B^{H}_1 \vert^m\right]t^{H(2-m)} \nonumber \\
			& = \hat{\mathbb{E}} \left[\vert B_1 \vert^m\right]t^{H(2-m)} \leq C_m t^{H(2-m)} \to 0 \quad \text{ for }t \to \infty,
		\end{align}
		where we used Exercise 3.10.1 in \cite{peng_nonlinearExpectation_book} for the last inequality \blue{and again Lemma \ref{lemma:DistributionFractionalGBM} for the last equality}. 
		This shows that in the one-dimensional case $B^H$ is a fractional $G$-Brownian motion in the sense of Definition \ref{def:fBrownianMotionSetting1}.
\end{proof}}

\subsection{Sample paths properties of \blue{a} multi-dimensional {$G$-fBm}}
From now on, we denote the Euclidean norm by $\lVert \cdot \rVert$.

\begin{prop} \label{prop:PathsHoelder}
{Let $B^H=(B^H_t)_{t \geq 0}$ be a $d$-dimensional fractional $G$-Brownian motion of Hurst \blue{index} $H \in \left( 0,1\right)$.} Consider $q \in \mathbb{R}^+$ such that {$q\geq 2H$}. Then for any $0<s<t$ it holds
\begin{align}\nonumber
	\hat{\mathbb{E}}\left[ \lVert B_t^H-B_s^H \rVert^{\frac{q}{H}}\right]  \le C_{q,H} (t-s)^q, 
\end{align}
for a given constant $C_{q,H}>0$.
{In particular, for any $\alpha < H$, $B^H$ admits a modification $\tilde{B}^H$ {which} has quasi-surely H\"older continuous paths of order $\alpha$.} 
\end{prop}
\begin{proof} 
	Fix $H \in \left(0,1\right)$, $0<s<t$ and $q\in \mathbb{R}^+$ such that {$q\geq 2H$}. We have
\begin{align}
	\hat{\mathbb{E}}\left[ \lVert B_t^H-B_s^H \rVert^{\frac{q}{H}}\right] &= \hat{\mathbb{E}}\left[ \left (\sum_{i=1}^d(B_t^H(i)-B_s^H(i))^2  \right)^{\frac{q}{2H}}\right] \nonumber\\
	&\blue{= \hat{\mathbb{E}}\left[ \left (\sum_{i=1}^d ( B_{t-s}^H(i))^2  \right)^{\frac{q}{2H}}\right]}\label{eq:Paths1Ddim} \\
	&= \hat{\mathbb{E}}\left[ \left (\sum_{i=1}^d ((t-s)^H B_1^H(i))^2  \right)^{\frac{q}{2H}}\right]\label{eq:Paths2Ddim} \\
	%&= \hat{\mathbb{E}}\left[ \left (\sum_{i=1}^d (B_1^H(i))^2  \right)^{\frac{q}{2H}}\right](t-s)^q \nonumber\\
	&= \hat{\mathbb{E}}\left[ \left (\sum_{i=1}^d (B_1(i))^2  \right)^{\frac{q}{2H}}\right](t-s)^q \label{eq:Paths3Ddim}\\
	&= \hat{\mathbb{E}}\left[  \lVert  B_1-B_0  \rVert^{\frac{q}{H}}\right](t-s)^q \nonumber\\
	&\le C_{q,H} (t-s)^q,\label{eq:Paths4Ddim}
\end{align}
where \eqref{eq:Paths1Ddim}, \eqref{eq:Paths2Ddim} and \eqref{eq:Paths3Ddim} follow by Proposition \ref{prop:statincr}, Proposition \ref{prop:SelfSimilar_and_StationarityIncrements} and from the results in the proof of Proposition \ref{prop:PDEFractionalBrownianMotion}, respectively. Moreover, we get 
\eqref{eq:Paths4Ddim} by equation (4.1) of  \cite{geng_qian_yang_2013} using  {$q\geq 2H$}.
\end{proof}
\begin{lemma} \label{lemma:ConvergenceLemmaSquared}
	{Let $B^H=(B^H_t)_{t \geq 0}$ be a $d$-dimensional fractional $G$-Brownian motion of Hurst \blue{index} $H \in \left( 0,1\right)$.} \blue{Let $\blue{\pi_t^N}=\lbrace t_0^N,\dots, t_{N}^N \rbrace$, $N \in \mathbb{N}$,} be a \blue{sequence of partitions} of $[0,t]$ such that the mesh size \blue{$\mu(\blue{\pi_t^N})$} converges to $0$ for $N \to \infty$. Then for $p>\frac{1}{H}$  it holds
	\begin{align} \label{eq:ConvegenceL1Capacity}
		 \lim_{\blue{N} \to \infty}\sum_{i=0}^{\blue{N}-1} \vert| B_{t_{i+1}^{\blue{N}}}^H-B_{t_i^{\blue{N}}}^H\vert|^p =0  \quad \text{ in } L^1 \text{ and  capacity.}
	\end{align}
\end{lemma}
\begin{proof}
	It is enough to show that $\lim_{\blue{N} \to \infty}\sum_{i=1}^{\blue{N}} \| B_{t_{i+1}^{\blue{N}}}^H-B_{t_i^{\blue{N}}}^H\|^p =0 $ in $L^1$ as $L^1$-convergence implies convergence in capacity by Point (c), page 5 in \cite{hu_zhou_2018}. It holds 
	\begin{align}
\hat{\mathbb{E}} \left[ \sum_{i=0}^{\blue{N}-1} \left \| B_{t_{i+1}^{\blue{N}}}^H-B_{t_i^{\blue{N}}}^H\right \|^p\right]	& \leq \sum_{i=0}^{\blue{N}-1} \hat{\mathbb{E}}  \left[ \left \| B_{t_{i+1}^{\blue{N}}}^H-B_{t_i^{\blue{N}}}^H\right \|^p\right]\nonumber \\
& \leq C_H \sum_{i=0}^{\blue{N}-1}  \left \vert  t_{i+1}^{\blue{N}} -t_i^{\blue{N}} \right \vert^{p{H}} \label{eq:L1Convergence1}  \\
& {=} C_H \blue{\cdot \left(\mu(\blue{\pi_t^N})\right)}^{pH-1} \sum_{i=0}^{\blue{N}-1} \left \vert  t_{i+1}^{\blue{N}} -t_i^{\blue{N}} \right \vert \left(\frac{\left \vert  t_{i+1}^{\blue{N}} -t_i^{\blue{N}} \right \vert}{\blue{\mu(\blue{\pi_t^N})}}\right)^{pH-1} \nonumber \\
& \leq C_H \blue{\cdot \left(\mu(\blue{\pi_t^N})\right)}^{pH-1} \sum_{i=0}^{\blue{N}-1} \left \vert  t_{i+1}^{\blue{N}} -t_i^{\blue{N}} \right \vert \nonumber \\
&\leq C_H  \blue{\cdot t \cdot \left(\mu(\blue{\pi_t^N})\right)}^{pH-1} \to 0 \quad \text{ for } \blue{N} \to \infty. \nonumber
\end{align}
where we use Proposition \ref{prop:PathsHoelder} for $	p=2H$ in \eqref{eq:L1Convergence1}. Moreover, the limit follows as $pH-1>0$.
\end{proof}
\begin{remark}
\begin{enumerate}
	\item Note that Lemma \ref{lemma:ConvergenceLemmaSquared} is in line with the well-known result that for a ``classical'' fractional Brownian motion $\bar{B}^H$ it holds 
	$$\lim_{n\to \infty} \sum_{i=0}^{\blue{N}-1} \|  \bar{B}_{t_{i+1}^{\blue{N}}}^H-\bar{B}_{t_i^{\blue{N}}}^H \|^{p}=0$$ 
	for $p>\frac{1}{H}$, where the limit is understood in a $L^1$-sense, see e.g. \cite{Guerra_Nualart_2005} and \cite{rogers_arbitrage_1997}. However, we cannot deduce any result for $p<\frac{1}{H}$ and $p=\frac{1}{H}$ as this would be based on ergodic theory, which is yet not available in the $G$-expectation framework and is beyond the scope of this paper.
	\item As in Section 3.2 in \cite{peng_zhang_2017}, we define the quadratic variation of a one-dimensional fractional $G$-Brownian motion $B^H$ as the $L^1$-limit in \eqref{eq:ConvegenceL1Capacity} for $p=2$. Then it follows by Lemma \ref{lemma:ConvergenceLemmaSquared} that for $H \in (\frac{1}{2},1)$ the quadratic variation process of \blue{a} fractional $G$-Brownian motion equals $0$. 
	\end{enumerate} 
\end{remark}

The next result shows that \blue{a} fractional $G$-Brownian motion has quasi-surely non-differentiable sample paths.
\begin{prop}
	{Let $B^H=(B^H_t)_{t \geq 0}$ be a $d$-dimensional fractional $G$-Brownian motion of Hurst \blue{index} $H \in \left( 0,1\right)$.} Then,
	\begin{equation*}
		c\left( \limsup_{t \to t_0} \bigg\lVert \frac{B_t^H-B_{t_0}^H}{t-t_0}\bigg\rVert<\infty\right)=0.
	\end{equation*}
\end{prop}
\begin{proof}
	We define the sets
	\begin{equation*}
		A^H(t):=\bigg \lbrace{ \sup_{0 \leq s \leq t } \bigg \Vert \frac{B^H_s(\omega)}{s} \bigg \Vert \leq d \bigg\rbrace}.
	\end{equation*}
	For any $(t_n)_{n \geq 0}$ with $t_n \downarrow 0$ \blue{for $n \to \infty$} it holds $A^{H}(t_{n}) \subset A^{H}(t_{n+1})$. Thus
	\begin{equation*}
		c\left(\lim_{n \to \infty} A^H(t_n)\right)=\lim_{n \to \infty} c\left(A^H(t_n)\right)
	\end{equation*}
	and we now prove that $\lim_{n \to \infty} c\left(A^H(t_n)\right)=0.$
	We have 
	\begin{align}
		\lim_{n \to \infty} c\left(A^H(t_n)\right)&=\lim_{n \to \infty} c \left( \sup_{0 \leq s \leq t_n} \bigg \Vert \frac{B^H_s(\omega)}{s} \bigg \Vert \leq d  \right) \nonumber \\
		&\leq \lim_{n \to \infty} c \left( \bigg \Vert \frac{B^H_{t_n}(\omega)}{t_n} \bigg \Vert \leq d  \right) \nonumber \\
		&=\lim_{n \to \infty} c \left( \Vert B_1 t_n^{H-1}\Vert \leq d \right) \label{eq:NowhereDifferentiable1} \\
		&=\lim_{n \to \infty} c \left( \Vert B_1 \Vert \leq d t_n^{1-H} \right) \nonumber \\
		&\leq \lim_{n \to \infty } \exp \left( \frac{1}{2} \right) (t_n^{1-H} d)^{2\alpha} \label{eq:NowhereDifferentiable2} \\
		&=0 \nonumber 
	\end{align}
	with $\alpha:=\frac{\underline{\sigma}^2}{2}$. 
	In \eqref{eq:NowhereDifferentiable1} we use Proposition \ref{prop:SelfSimilar_and_StationarityIncrements} and in \eqref{eq:NowhereDifferentiable2} Lemma 3.6 in \cite{Wang_Zheng_2018}.
\end{proof}
\section{{Stochastic calculus} for  \blue{a multi-dimensional} $G$-fBm \blue{with} $H \in (\frac{1}{2},1)$} \label{sec:StochasticCalculus}
{In this section we introduce the definition of a stochastic integral with respect to a $d$-dimensional fractional $G$-Brownian motion \greenNew{of Hurst index $H \in (\frac{1}{2},1)$} by following the approach of \cite{zaehle}.}
\subsection{Fractional calculus and deterministic differential equations}
In order to define the generalized Lebesgue-Stieltjes integral as in \cite{zaehle} we introduce some notation about fractional calculus.
For a finite interval $[a,b] \subset \mathbb{R}$ and $1 \leq p \leq \infty$ we set
\begin{equation*}
L^p(a,b):= \lbrace{ f: [a,b] \to \mathbb{R}^n: \ f  \ \text{{Borel} mesurable, } \| f \|_{L^p([a,b])}<\infty \rbrace},
\end{equation*}
where 
\begin{align*}
	\| f \|_{L^p(a,b)}&:= \left( \int_a^b \vert f(x) \vert^p dx \right)^{\frac{1}{p}}, \quad 1 \leq p < \infty ,\\
	\| f \|_{L^{\infty}(a,b)}&:=\textnormal{ess sup} \lbrace \vert f (x) \vert: x \in [a,b]\rbrace.
\end{align*}
{For any $0 < \lambda \leq 1$ we define $C^{\lambda}(0,T; \mathbb{R}^d)$ as the space of H\"older continuous functions $f:[0,T] \to \mathbb{R}^d$ with exponent $\lambda$.}
\begin{defi}
Let $f \in L^1(a,b)$ and $\alpha >0$. The \emph{left-sided fractional Riemann-Louiville integral of $f$ of order $\alpha$} is given for almost all $x \in (a,b)$ by
\begin{equation*}
	I_{a+}^{\alpha}f(x):=\frac{1}{\Gamma(\alpha)} \int_a^x \frac{f(y)}{(x-y)^{1-\alpha}}dy,
\end{equation*}
whereas the \emph{{right}-sided fractional Riemann-Louiville integral of $f$ of order $\alpha$} is given for almost all $x \in (a,b)$ by
\begin{equation*}
	I_{b+}^{\alpha}f(x):=\frac{(-1)^{-\alpha}}{\Gamma(\alpha)} \int_x^b \frac{f(y)}{(y-x)^{1-\alpha}}dy,
\end{equation*}
where $(-1)^{-\alpha}=e^{-i \pi \alpha}$ and $\Gamma(\alpha)=\int_0^{\infty}r^{\alpha-1}e^{-r}dr. $
\end{defi}
The class $I_{+}^{\alpha}(L^p(a,b))$ contains all functions $f$ that can be represented as $f=I_{a+}^{\alpha}\varphi$ for $\varphi \in L^p(a,b)$ and $I_{-}^{\alpha}(L^p(a,b))$ contains all functions $f$ that can be represented as $f=I_{a-}^{\alpha}\varphi$ for $\varphi \in L^p(a,b)$, respectively.  Moreover, for $f \in I_{a+}^{\alpha}(L^p(a,b))$  (respectively $f \in I_{a-}^{\alpha}(L^p(a,b))$) and $0< \alpha < 1$ we introduce the \emph{Weyl derivatives} by
\begin{align*}
	D_{a+}^{\alpha}f(x)&:= \frac{1}{\Gamma(1-\alpha)} \left(\frac{f(x)}{(x-a)^{\alpha}}+\alpha \int_a^x \frac{f(x)-f(y)}{(x-y)^{\alpha +1}}dy \right) \textbf{1}_{(a,b)}(x), \\
	D_{b-}^{\alpha}f(x)&:= \frac{(-1)^{\alpha}}{\Gamma(1-\alpha)} \left(\frac{f(x)}{(b-\blue{x})^{\alpha}}+\alpha\int_x^b \frac{f(x)-f(y)}{(y-x)^{\alpha +1}}dy \right)\textbf{1}_{(a,b)}(x), 
\end{align*}
respectively, where the convergence of the integral holds pointwise for almost all $x \in (a,b)$ for $p=1$ and in $L^p(a,b)$ for $p>1$.\\
For a function $f:[a,b] \to \mathbb{R}$ we use the notation 
\begin{align} \label{eq:Limit1}
	f(u+):= \lim_{\delta \blue{\downarrow} 0}f(u + \delta) \quad \text{ and } \quad f(u-):=\lim_{\delta \blue{\downarrow} 0} f(u-\delta) \text{ for } a \leq u \leq b,
\end{align}
and
\begin{align} \label{eq:Limit2}
	f_{a+}(x):=(f(x)-f(a+))\textbf{1}_{(a,b)}(x) \quad \text{ and } \quad f_{b-}(x):=(\blue{f(x)-f(b-)}) \textbf{1}_{(a,b)}(x),
\end{align}
where we assume that the limits in \eqref{eq:Limit1} and \eqref{eq:Limit2} exist and are finite.
\begin{defi} \label{def:GeneralizedLebegueStieltjes}
Let $f,g: [a,b] \to \mathbb{R}$ such that $f(a+),g(a+),g(b-)$ exist, and $f_{a+} \in I_{a+}^{\alpha} (L^p([a,b]))$, $g_{b-} \in I_{b+}^{1-\alpha} (L^	q([a,b]))$ for some $p,q \geq 1$ with $\frac{1}{p}+\frac{1}{q}\leq 1$ and $0 \leq \alpha \leq 1$. Then the \emph{generalized fractional Lebesgue-Stieltjes integral} is defined by
\begin{equation}\label{eq:DefLebStieltjes}
	\int_a^b f(x) dg(x):=(-1)^{\alpha} \int_a^b D^{\alpha}_{a+} f_{a+}(x)D_{b-}^{1-\alpha}g_{b-}(x)dx + f(a+)(g(b-)-g(a+)).
\end{equation}
\end{defi}
\begin{remark} \label{remark:GeneralizedLebesgueStieltjes}
For $f \in C^{\lambda}([a,b]), g \in C^{\mu}([a,b])$ with $\mu+ \lambda >1${,} the conditions in Definition \ref{def:GeneralizedLebegueStieltjes} are satisfied and we can choose $p=q=\infty$ and $\alpha <\lambda, 1-\alpha<\mu$. The integral of $\int_a^b f(x)dg(x)$ in  \eqref{eq:DefLebStieltjes} coincides then with the Riemann-Stieltjes integral of $f$ with respect to $g$. 
\end{remark}
For $0<\alpha<1 / 2$ we denote by $W_T^{1-\alpha, \infty}(0, T)$ the space of measurable functions $g:[0, T] \rightarrow \mathbb{R}$ such that
$$
\|g\|_{1-\alpha, \infty, T}:=\sup _{0<s<t<T}\left(\frac{|g(t)-g(s)|}{(t-s)^{1-\alpha}}+\int_s^t \frac{|g(y)-g(s)|}{(y-s)^{2-\alpha}} d y\right)<\infty .
$$
In particular, it holds
\begin{align} \label{eq:InclusionsSpaces}
C^{1-\alpha+\varepsilon}(0, T) \subset W_T^{1-\alpha, \infty}(0, T) \subset C^{1-\alpha}(0, T), \quad \forall \varepsilon>0 .
\end{align}
For $0 < \alpha<\frac{1}{2}$, $W_0^{\alpha, 1}(0, T)$ {is defined as the space of} measurable functions $f$ on $[0, T]$ such that
$$
\|f\|_{\alpha, 1}:=\int_0^T \frac{|f(s)|}{s^\alpha} d s+\int_0^T \int_0^s \frac{|f(s)-f(y)|}{(s-y)^{\alpha+1}} d y d s<\infty .
$$
Note that the restriction of $ g \in W_T^{1-\alpha, \infty}(0, T)$ to $(0, t)$ belongs to $I_{t-}^{1-\alpha}\left(L^{\infty}(0, t)\right)$ for all \blue{$t \in [0,T]$} and the restriction of $f \in W_0^{\alpha, 1}(0, T)$ to $(0, t)$ belongs to $I_{0+}^\alpha\left(L^1(0, t)\right)$ for all  \blue{$t \in [0,T]$}.
Furthermore, for any  $f \in W_0^{\alpha, 1}(0, T)$ and $g\in W_T^{1-\alpha, \infty}(0, T)$,  the integral $\int_0^t f d g$ exists for all $t \in[0, T]$ and $$
\int_0^t f d g=\int_0^T f \mathbf{1}_{(0, t)} d g.
$$
{For $0<\alpha < \frac{1}{2}$ we denote by $W_0^{\alpha,\infty}(0,T;\mathbb{R}^d)$ the space of all measurable functions $f:[0,T] \to \mathbb{R}^d$ such that
\begin{align*}
	\| f \|_{\alpha, \infty}:= \sup_{t \in [0,T]} \left( \vert f (t) \vert + \int_0^t \frac{\vert f(t)-f(s) \vert}{(t-s)^{\alpha+1}} ds \right)< \infty.
\end{align*}}
\begin{asum} \label{assump:SDE}
Fix $m \in \mathbb{N}$. For $i=1,...,d$ and $j=1,...,m$ let \linebreak $\sigma^{i,j},b^i:{([0,T] \times \mathbb{R}^d, \mathcal{B}([0,T] \times \mathbb{R}^d)) \to (\mathbb{R},\mathcal{B}(\mathbb{R}))}$ be measurable functions. {Given} $\sigma:[0,T] \times \mathbb{R}^d \to \mathbb{R}^{d \times m}$ and $b:[0,T] \times \mathbb{R}^d \to \mathbb{R}^d$ by $\sigma(t,x):=\left( \sigma^{i,j}(t,x)\right)_{i=1,...,d; j=1,...,m}$ and $b(t,x):=(b^i(t,x))_{i=1,...,d}${, we assume that}
\begin{enumerate}
	\item $\sigma(t, x)$ is differentiable in $x$,  there exist some constants $0<\beta, \delta \leq 1,$ and for every $N \geq 0$ there exists $M_N>0$ such that the following properties hold for each $i=1, \ldots, d$
\begin{enumerate}
\item  $\sigma$ is Lipschitz continuous \blue{in space}, i.e. for all  $ x \in \mathbb{R}^d, t \in[0, T]$ it holds 
$$\|\sigma(t, x)-\sigma(t, y)\| \leq M_0\|x-y\|\blue{;}$$
\item $\sigma$ is local H\"older continuous \blue{in space}, i.e. for all $\|x\|,\|y\| \leq N,  t \in[0, T]$, it holds
$$\|\partial_{x_i} \sigma(t, x)-\partial_{y_i} \sigma(t, y)\| \leq M_N\|x-y\|^\delta\blue{;}$$
\item $\sigma$ is H\"older continuous in time, i.e. for all $ x \in \mathbb{R}^d, t, s \in[0, T]$ it holds
$$\|\sigma(t, x)-\sigma(s, x)\|+\|\partial_{x_i} \sigma(t, x)-\partial_{x_i} \sigma(s, x)\| \leq M_0|t-s|^\beta.$$
\end{enumerate}
\item 
\blue{T}here exists $b_0 \in L^\rho\left(0, T ; \mathbb{R}^d\right)$, where $\rho \geq 2$, and for every $N \geq 0$ there exists $L_N>0$ such that the following properties hold
\begin{enumerate}
	\item $b$ is local Lipschitz continuous \blue{in space}, i.e. for all $ \|x\|,\|y\| \leq N, t \in [0,T]$ it holds $$\|b(t, x)-b(t, y)\| \leq L_N\|x-y\|\blue{;}$$
	\item $b$ is bounded, i.e. for all $x \in \mathbb{R}^d, t \in [0,T]$ it holds $$\|b(t, x)\| \leq L_0\|x\|+b_0(t).$$ 
\end{enumerate}
\end{enumerate}
Here we use the notation 
$$
\| y \|:= \left(\sum_{i=1}^d \vert y^i \vert^2\right)^{\frac{1}{2}} \quad  \text{ and } \quad \| A \|:=\left(\sum_{i=1}^d \sum_{j=1}^m \vert a^{i,j} \vert^2 \right)^{\frac{1}{2}}
$$
for $y \in \mathbb{R}^d$ and $A=(a^{i,j})_{i=1,...,d; j=1,...,m} \in \mathbb{R}^{d \times m}.$
\end{asum}
We now state Theorem 5.1 in \cite{nualart2002}.
\begin{theorem} \label{theorem:ExistenceSDE}
	Let $0<\alpha<\frac{1}{2}$ be fixed. Let $g \in W_T^{1-\alpha, \infty}\left(0, T ; \mathbb{R}^m\right)$. Consider the deterministic differential equation on $\mathbb{R}^d$
\begin{align} \label{eq:DifferentialEquation}
x_t^i=x_0^i+\int_0^t b^i\left(s, x_s\right) d s+\sum_{j=1}^m \int_0^t \sigma^{i, j}\left(s, x_s\right) d g_s^j, \quad t \in[0, T],
\end{align}
$i=1, \ldots, d$, where $x_0 \in \mathbb{R}^d$, and the coefficients $\sigma^{i, j}, b^i:[0, T] \times \mathbb{R}^d \rightarrow \mathbb{R}$ are measurable functions satisfying Assumption \ref{assump:SDE}, respectively with $\rho=\frac{1}{\alpha}, 0<$ $\beta, \delta \leq 1$ and
$$
0<\alpha<\alpha_0=\min \left\{\frac{1}{2}, \beta, \frac{\delta}{1+\delta}\right\}.
$$
Then, equation \eqref{eq:DifferentialEquation} has a unique solution $x \in W_0^{\alpha,\infty}(0,T;\mathbb{R}^d)$, which is $(1-\alpha)$-H\"older continuous. 
\end{theorem}

\subsection{Stochastic integrals with respect to \blue{a multi-dimensional} {$G$-fBm} \blue{with} $H \in (\frac{1}{2},1)$}
From now on we fix $H \in (\frac{1}{2},1)$ \blue{and $T>0$}. {{We denote by $L^0\left(\Omega, \mathcal{F}, \mathcal{P}; W_0^{\alpha,\infty}(0,T)\right)$ the space of functions $X$ \blue{on $(\Omega, \mathcal{F})$} such that $X \in W_0^{\alpha,\infty}(0,T)$ quasi surely. Analogously, we define the space $L^0\left(\Omega, \mathcal{F}, \mathcal{P}; W_T^{\alpha, 1}(0,T)\right)$.} Here, {$\mathcal{F}=\mathcal{B}(\Omega)$} and $\mathcal{P}$ \blue{are} given in \eqref{eq:GExpectationUpperExpectation}. From now on\blue{,} we will call {the}  elements of the spaces above \emph{stochastic processes} with a small abuse of terminology.  
{With a similar notation, $L^{\infty}\left(\Omega, \mathcal{F}, \mathcal{P} ; \mathbb{R}^d\right)$ is the space of functions {$X:(\Omega, \mathcal{F}) \to \mathbb{R}^d$ such that $\| X(\omega)\|_{L^\infty}<\infty $ for quasi every $\omega$. }
} 

 First note that by Proposition \ref{prop:PathsHoelder} \blue{a $d$-dimensional fractional $G$-Brownian motion} $B^H$ is quasi-surely H\"older continuous for any $\beta < H$. Thus, by \eqref{eq:InclusionsSpaces} it follows that quasi-surely for any $T>0$ the trajectories of $B^H$ belong to $W_T^{1-\alpha, \infty}(0,T)$ for $1-H <\alpha <\frac{1}{2}$. \blue{From now on, we fix \greenNew{$\alpha \in (1-H,\frac{1}{2})$}.}
\begin{defi} \label{def:IntegralFractionalGBrownianMotion}
{Let $B^H=(B^H_t)_{\blue{t \in [0,T]}}$ be a $d$-dimensional fractional $G$-Brownian motion of Hurst \blue{index} $H \in \left(\frac{1}{2},1\right)$.} Let {$u=(u_t)_{t \in [0,T]} \in L^0\left(\Omega, \mathcal{F}, \mathcal{P}; W_T^{\alpha, 1}(0,T)\right)$}. Then we define the pathwise integral
\begin{equation} \label{eq:DefinitionIntegralfGBMPathwise}
\int_0^t u_s d B^H_s, \quad t \in [0,T]
\end{equation}
in the sense of Definition \ref{def:GeneralizedLebegueStieltjes}.
\end{defi}
We now provide an inequality for the pathwise increments of a fractional $G$-Brownian motion, which is a generalization of Lemma 7.4 in \cite{nualart2002}. 
\begin{lemma}\label{lemma:firstnualart}
	Let $B^H=(B_t^H)_{t \in [0,T]}$ be a {$d$-dimensional} fractional $G$-Brownian motion \blue{of Hurst index} $H \in (\frac{1}{2},1)$. Then for every $0<\varepsilon<H$ and $T>0$ there exists a positive {one-dimensional} random variable $\eta_{\varepsilon,H, T}$ with
	$$
	\hat{\mathbb{E}}\left[\left|\eta_{\varepsilon, H, T}\right|^p\right]<\infty
	$$
	for all $p \in[1, \infty)$ and such that for all $s, t \in[0, T]$ it holds
$$
\|B_t^H-B_s^H\| \leq \eta_{\varepsilon, H, T}|t-s|^{H-\varepsilon} \quad \text { q.s. }
$$
\end{lemma}
\begin{proof}
By Lemma 7.3 in \cite{nualart2002} {and since $B^H$ has quasi-surely continuous paths,} it holds quasi-surely for $s, t \in[0, T]$
\begin{align} \label{eq:DifferencePathwise}
\|B_t^H-B_s^H\| \leq C_{H, \varepsilon}|t-s|^{H-\varepsilon} \xi
\end{align}
with
\begin{equation}
\xi=\left(\int_0^T \int_0^T \frac{\|B_r^H-B_{v}^H\|^{\frac{2}{\varepsilon}}}{|r-v|^{\frac{2H}{ \varepsilon}}} d r d v\right)^{\frac{\varepsilon}{2}} .\label{eq:xi}
\end{equation}
\allowdisplaybreaks
Fix $q \geq 2 / \varepsilon$. Then we get
\begin{align} 
	\hat{\mathbb{E}} \left[ \vert \xi \vert^q\right] &= \hat{\mathbb{E}} \left[ \left(\int_0^T \int_0^T \frac{\|B_r^H-B_{v}^H\|^{\frac{2}{\varepsilon}}}{|r-v|^{\frac{2H}{\varepsilon}}} d r d v\right)^{\frac{q \varepsilon}{2}} \right] \nonumber \\
	&=\sup_{P \in \mathcal{P}} {E}^P \left[ \left(\int_0^T \int_0^T \frac{\|B_r^H-B_{v}^H\|^{\frac{2}{\varepsilon}}}{|r-v|^{\frac{2 H}{\varepsilon}}} d r d v\right)^{\frac{q \varepsilon}{2}} \right] \label{eq:MinkowskiInequality1} \\
	&\leq\left( \sup_{P \in \mathcal{P}} \int_0^T \int_0^T  \left(\int_{\Omega} \left( \frac{\|B_r^H(\omega)-B_{v}^H(\omega)\|^{\frac{2}{\varepsilon}}}{|r-v|^{\frac{2H}{\varepsilon}}} \right)^{\frac{q \varepsilon}{2}} dP(\omega) \right)^{\frac{2}{q \varepsilon}}  d r d v \right)^{\frac{q\epsilon}{2}}\label{eq:MinkowskiInequality2} \\
	&\leq\left( \int_0^T \int_0^T  \frac{1}{|r-v|^{\frac{2H}{\varepsilon}}}\sup_{P \in \mathcal{P}}\left(\int_{\Omega} \left( \|B_r^H(\omega)-B_{v}^H(\omega)\|^{\frac{2}{\varepsilon}} \right)^{\frac{q \varepsilon}{2}} dP(\omega) \right)^{\frac{2}{q \varepsilon}}  d r d v \right)^{\frac{q\epsilon}{2}}\label{eq:MinkowskiInequality3} \\
	&= {\left( \int_0^T \int_0^T  \frac{1}{|r-v|^{\frac{2H}{\varepsilon}}}\left(\hat{\mathbb{E}}\left[ \|B_r^H(\omega)-B_{v}^H(\omega)\|^{q} \right]\right)^{\frac{2}{q \varepsilon}}  d r d v \right)^{\frac{q\epsilon}{2}}} \nonumber \\
	& \leq {\left( \int_0^T \int_0^T  \frac{1}{|r-v|^{\frac{2H}{\varepsilon}}}\left(C_{\varepsilon,H}\vert r- v \vert^{qH}\right)^{\frac{2}{q \varepsilon}}  d r d v \right)^{\frac{q\epsilon}{2}}} \label{eq:MinkowskiInequality4} \\
	& \leq {C_{\varepsilon,H} \left( \int_0^T \int_0^T    d r d v \right)^{\frac{q\epsilon}{2}}}\nonumber \\
	& \leq C_{\varepsilon,H} T^{q \varepsilon},\label{eq:EsitmateXi}
	\end{align}
where we use in \eqref{eq:MinkowskiInequality1} the representation of the $G$-expectation in \eqref{eq:GExpectationUpperExpectation}. Inequality \eqref{eq:MinkowskiInequality2} follows by the Minkowski inequality. {The inequality in} \eqref{eq:MinkowskiInequality3} {follows} since 
$
\sup_{P \in \mathcal{P}} E^P\left[\left( \|B_r^H(\omega)-B_{v}^H(\omega)\|^{\frac{2}{\varepsilon}} \right)^{\frac{q \varepsilon}{2}}\right ]
$ is measurable. Moreover, Proposition \ref{prop:PathsHoelder} yields \eqref{eq:MinkowskiInequality4}. By setting 
\begin{equation*}
\eta_{\epsilon,T,H}=C_{\varepsilon, H} \xi
\end{equation*}
 the result follows by \eqref{eq:DifferencePathwise} and \blue{\eqref{eq:EsitmateXi}}.
\end{proof}

\begin{lemma}
	Let $B^H=(B_t^H)_{\blue{t \in [0,T]}}$ be a fractional $G$-Brownian motion {of Hurst \blue{index} $H \in \left( \frac{1}{2},1\right)$} and $u=(u_t)_{t \in [0,T]} \in L^0\left(\Omega, \mathcal{F}, \mathcal{P}; W_T^{\alpha, 1}(0,T)\right)$. Then, the pathwise integral defined in \eqref{eq:DefinitionIntegralfGBMPathwise} satisfies 
\begin{equation} \label{eq:BoundIntegral}
\left|\int_0^Tu_sdB^H_s \right| \le \tilde{G} \lVert u \rVert_{\alpha,1} \quad \text{{quasi-surely}},
\end{equation}
where $\tilde{G}$ is defined by
\begin{equation}\label{eq:nualartG}
\tilde{G}:=\frac{1}{\Gamma(1-\alpha)}\sup_{0< s< t< T} \left|\left( D_{t-}^{1-\alpha}B^{H}_{t-}\right)(s)\right|,
\end{equation}
and has {finite} moments of all orders for $p \in [1,\infty)$.
If $u \in  L^0\left(\Omega, \mathcal{F}, \mathcal{P}; W_0^{\alpha, \infty}(0,T)\right)$, then the integral $(\int_0^t u_sdB_s^H)_{0 \le t \le T}$ is H\"older continuous of order $1-\alpha$ {quasi-surely.}
\end{lemma}
{\begin{proof}
Following the same steps of the proof of Lemma 7.5 in \cite{nualart2002}, it can be seen that Lemma {\ref{lemma:firstnualart}} implies that, if $1-H<\alpha<\frac{1}{2}$, then
\begin{equation}\label{eq:lemma75nualart}
\hat{\mathbb{E}}\left[\sup_{0 < s < t < T}\left|D_{t-}^{1-\alpha}B_{t-}^{H}(s)\right|^p\right]<\infty
\end{equation}
for all $T>0$ and $p \in [1,\infty)$. In particular, \eqref{eq:lemma75nualart} implies that the random variable $\tilde{G}$ defined in \eqref{eq:nualartG} has moments of all orders for $p \in [1,\infty)$. Thus, {the inequality in \eqref{eq:BoundIntegral} follows by} equation (4.10) in \cite{nualart2002}. The H\"older continuity of the integral $(\int_0^t u_sdB_s^H)_{0 \le t \le T}$ for $u \in  L^0\left(\Omega, \mathcal{F}, \mathcal{P}; W_0^{\alpha, \infty}(0,T)\right)$ {is a consequence of} by Proposition 4.1 in \cite{nualart2002}.  
\end{proof}}
\begin{theorem}
Let $B^H=(B^H_t)_{t \in [0,T]}$ be a $d$-dimensional fractional $G$-Brownian motion of Hurst \blue{index} $H \in \left(\frac{1}{2},1\right)$ defined on {$(\Omega,\mathcal{F}, \mathcal{P}).$}
We consider 
\begin{equation} \label{eq:SDE}
X_t^i=X_0^i+\sum_{j=1}^m \int_0^t \sigma^{i, j}\left(s, X_s\right) d B_s^H(j)+\int_0^t b^i\left(s, X_s\right) d s, \quad t \in[0, T],
\end{equation}
$i=1, \ldots, d$, where {$X_0:(\Omega, \mathcal{F}) \to (\mathbb{R}^d, \mathcal{B}(\mathbb{R}^d))$ is a measurable function.} {Here we assume that 
$$
\sigma^{i, j}, b^i: \Omega \times[0, T] \times \mathbb{R}^d \rightarrow \mathbb{R}, \quad i=1,...,d, \quad j=1,...,m
$$
are such that for quasi all $\omega \in \Omega$, the functions $\sigma^{\omega}: [0,T] \times \mathbb{R}^d \to \mathbb{R}^{d \times m}$ and $b^{\omega}: [0,T] \times \mathbb{R}^d \to \mathbb{R}^{d}$ defined by $\sigma^{\omega}(t,x):=\left( \sigma^{i,j}(\omega,t,x)\right)_{i=1,...,d; j=1,...,m}$ and $b^{\omega}(t,x):=(b^i(\omega,t,x))_{i=1,...,d}$ {are measurable and} satisfy points 1. and 2. of Assumption \ref{assump:SDE}} for $\beta >1 - H, \delta >\frac{1}{H}-1$, where the constants $M_N, L_N, K_0$ and the function $b_0$ may depend on $\omega$.
If $\alpha \in\left(1-H, \alpha_0\right)$ and $\rho \geq 1 / \alpha$, there exists a unique {solution} $X \in {L^0\left(\Omega, \mathcal{F}, \mathcal{P} ; W_0^{\alpha, \infty}\left(0, T ; \mathbb{R}^d\right)\right)}$ of the stochastic differential equation \eqref{eq:SDE}. Moreover, for quasi all $\omega \in \Omega$
$$
X(\omega, \cdot)=\left(X(\omega, \cdot)(i)\right)_{i=1,...,d} \in C^{1-\alpha}\left(0, T ; \mathbb{R}^d\right).
$$
\end{theorem}

\subsection{Pathwise It\^{o}'s formula for \blue{a one-dimensional} $G$-fBm \blue{with} $H \in (\frac{1}{2},1)$}

Next, we present an It\^{o} formula for \blue{a} fractional $G$-Brownian motion  for $H > \frac{1}{2}$ involving the integral introduced in \eqref{eq:DefinitionIntegralfGBMPathwise}. This is {an extension} of the main result of \cite{shiryaev1998arbitrage} to the $G$-setting {under an additional \blue{assumption} on the function $f$}.

\begin{prop} \label{prop:Ito}
	Let $B^H=(B_t^H)_{\blue{t \in [0,T]}}$ be a one-dimensional fractional $G$-Brownian motion \blue{of} Hurst index $H \in \left( \frac{1}{2},1 \right)$.
	\begin{enumerate}
		\item {Let $f \in C^2(\mathbb{R})$ with bounded second derivative and fix} \blue{$t \in [0,T]$}. Then it holds
		\begin{equation} \label{eq:Ito1}
		f(B_t^H)=f(0)+ \int_0^t \frac{\partial }{\partial x}f(B_u^H)dB_u^H,
	\end{equation}
	\blue{where the integral with respect to $B^H$ is given according to Definition \ref{def:IntegralFractionalGBrownianMotion}.}
		\item {Let} $f \in C^{1,2}(\blue{[0,T]}\times \mathbb{R})$ {such that  $\frac{\partial^2}{\partial^2 x}f$ is bounded} and $\int_0^t\left| \frac{\partial f}{\partial t}(s,B^H_s)\right|ds<\infty$ quasi-surely for any $t \blue{\in [0,T]}$ .  
		Then for any $\blue{t \in [0,T]}$ it holds
		\begin{equation} \label{eq:Ito2}
		f(t,B_t^H)=f(0,0)+ \int_0^t \frac{\partial}{\partial t}f(u,B_u^H)du + \int_0^t \frac{\partial}{\partial x} f(u,B_u^H)dB_u^H.
	\end{equation}
	{Here} $\frac{\partial}{\partial t}f$ and $\frac{\partial}{\partial x}f$  denote the partial derivatives of $f$ with respect to $t$ and $x$, respectively, {and $\frac{\partial^2}{\partial^2 x}f$ the second partial derivative of $f$ with respect to $x$.}
	\end{enumerate}
\end{prop} 
\begin{proof}
\blue{We start by proving Point 1. Let $\blue{\pi_{t}^N}=\lbrace t_0^{\blue{N}},...,t_{\blue{N}}^{\blue{N}} \rbrace$, $\blue{N \in \mathbb{N}}$, be a sequence of partitions of $[0,t]$ such that the mesh size \blue{$\mu(\pi_{t}^N)$ defined in \eqref{eq:MeshSize}} converges to $0$ for $\blue{N} \to \infty$.} By applying Taylor\blue{'s theorem} we get
\begin{align}
&f(B_t^H)-f(B_0^H)\nonumber \\ &= \sum_{i=0}^{\blue{N}-1}	 f(B_{t_{i+1}^{\blue{N}}}^H)-f(B_{t_i^{\blue{N}}}^H) \nonumber \\
& = \sum_{i=0}^{\blue{N}-1} \frac{\partial}{\partial x}f(B_{t_{i+1}^{\blue{N}}}^H)\left( B_{t_{i+1}^{\blue{N}}}^H-B_{t_i^{\blue{N}}}^H\right)+ \sum_{i=0}^{\blue{N}-1}\int_{B_{t_{i+1}^{\blue{N}}}}^{B_{t_i^{\blue{N}}}^H} \frac{\partial^2}{\partial x^2}f(u)(B_{t_i^{\blue{N}}}^H-u)du \nonumber \\
& \leq \sum_{i=0}^{\blue{N}-1}  \frac{\partial}{\partial x}f(B_{t_{i+1}^{\blue{N}}}^H)\left( B_{t_{i+1}^{\blue{N}}}^H-B_{t_i^{\blue{N}}}^H\right)+ \frac{1}{2} \sup_{u \in [0,t]} \left \vert  \frac{\partial^2}{\partial x^2}f(B_u^H)\right\vert \sum_{i=0}^{\blue{N}-1}\left \vert B_{t_{i\blue{+1}}^{\blue{N}}}^H-B_{t_i^{\blue{N}}}^H \right \vert^2.  \label{eq:Taylor}
\end{align}
By Lemma \ref{lemma:ConvergenceLemmaSquared} we have that
\begin{equation}\label{eq:convergencecapacity}
\lim_{\blue{N}\to \infty} {\sum_{i=0}^{\blue{N}-1}}\left \vert B_{t_{i\blue{+1}}^{\blue{N}}}^H-B_{t_i^{\blue{N}}}^H \right \vert^2=0 \text{ in capacity,}
\end{equation} 
and since $\frac{\partial^2}{\partial x^2}f$ is bounded by assumption this implies that the second term in \eqref{eq:Taylor} also converges to $0$ in capacity. \blue{We now focus on the first integral. }
\blue{First, we prove that $\frac{\partial }{\partial x}f\left(B_{\cdot}^H\right) \in L^0\left(\Omega, \mathcal{F}, \mathcal{P}; W_T^{\alpha, 1}(0,T)\right)$ \blue{i.e.} that for quasi-all $\omega \in \Omega$ it holds
\begin{align}
&\left\|\frac{\partial }{\partial x}f\left(B_T^H(\omega)\right)\right \|_{\alpha,1}\nonumber \\
&=\int_0^T \frac{\left|\frac{\partial }{\partial x}f\left(B_s^H(\omega)\right)\right|}{s^{\alpha}} d s+\int_0^T \int_0^s \frac{\left|\frac{\partial }{\partial x}f\left(B_s^H(\omega)\right)-\frac{\partial }{\partial x}f\left(B_y^H(\omega)\right)\right|}{(s-y)^{\alpha+1}} d y d s<\infty. \label{eq:normtocheck}
\end{align}
We have
\begin{align}
\int_0^T \frac{\left|\frac{\partial }{\partial x}f\left(B_s^H(\omega)\right)\right|}{s^{\alpha}} d s\leq \sup_{s \in [0, T]}\left|\frac{\partial }{\partial x}f\left(B_s^H(\omega)\right)\right| \int_0^T \frac{1}{s^{\alpha}} d s<\infty \text{ for quasi-all $\omega \in \Omega$} \label{eq:firstfiniteintegral}
\end{align}
since \greenNew{$\alpha <1$,} $f \in C^2(\mathbb{R})$ and $B^H$ has quasi-surely continuous paths.
Moreover, \allowdisplaybreaks
\begin{align}
&\int_0^T \int_0^s \frac{\left|\frac{\partial }{\partial x}f\left(B_s^H(\omega)\right)-\frac{\partial }{\partial x}f\left(B_y^H(\omega)\right)\right|}{(s-y)^{\alpha+1}} d y d s\nonumber \\
&=\int_0^T \int_0^s \frac{\left|\frac{\partial }{\partial x}f\left(B_s^H(\omega)\right)-\frac{\partial }{\partial x}f\left(B_y^H(\omega)\right)\right|}{\left|B_s^H(\omega)-B_y^H(\omega)\right|}\frac{\left|B_s^H(\omega)-B_y^H(\omega)\right|}{(s-y)^{\alpha+1}} d y d s\nonumber \\
&\le\int_0^T \int_0^s \left|\frac{\partial^2 }{\partial x^2}f(z_{\omega})\right|\frac{\left|B_s^H(\omega)-B_y^H(\omega)\right|}{(s-y)^{\alpha+1}} d y d s \quad \text{for $z_{\omega} \in [\inf_{r \in [s,y]}B_r(\omega), \sup_{r \in [s,y]}B_r(\omega)]$}\label{eq:centrallimittheorem} \\ 
&\le C \int_0^T \int_0^s\frac{\left|B_s^H(\omega)-B_y^H(\omega)\right|}{(s-y)^{\alpha+1}} d y d s \label{eq:secondderivativeboudned} \\ 
&< \infty \text{ for quasi-all $\omega \in \Omega$}, \label{eq:integralquasisurelyfinite}
\end{align}
where \eqref{eq:centrallimittheorem} is implied by the Mean Value Theorem, \greenNew{and} \eqref{eq:secondderivativeboudned} follows from the assumption that $\frac{\partial^2 }{\partial x^2}f$ is bounded. Finally, we get \eqref{eq:integralquasisurelyfinite} since the trajectories of $B^H$ belong to $W_T^{1-\alpha, \infty}(0,T)$ and $\alpha<\frac{1}{2}$. From \eqref{eq:firstfiniteintegral} and \eqref{eq:integralquasisurelyfinite} we conclude that \eqref{eq:normtocheck} holds for quasi-all $\omega \in \Omega$, so that $\frac{\partial }{\partial x}f\left(B_{\cdot}^H\right) \in L^0\left(\Omega, \mathcal{F}, \mathcal{P}; W_T^{\alpha, 1}(0,T)\right)$.
}
\blue{\greenNew{By} Definition \ref{def:IntegralFractionalGBrownianMotion} \greenNew{we obtain that}
$$
\int_0^t \frac{\partial }{\partial x}f(B_u^H)dB_u^H=\lim_{n \to \infty} \sum_{i=0}^{\blue{N}-1}  \frac{\partial}{\partial x}f(B_{t_{i+1}^{\blue{N}}}^H)\left( B_{t_{i+1}^{\blue{N}}}^H-B_{t_i^N}^H\right),
$$
\greenNew{as pathwise limit} by Remark \ref{remark:GeneralizedLebesgueStieltjes}.
}The proof of Point 2. follows similarly.
\end{proof}

{\begin{remark}
{The boundedness assumption for the second derivative on $f$ is necessary, as the standard techniques used in the classical case as in \cite{mishura_fractional_BM} cannot be applied in the sublinear $G$-setting. For example, the $G$-expectation does not satisfy the ``monotone continuity criteria'', see \cite{hu_zhou_2018} for more details on this. Note that it is also not possible  to prove Proposition \ref{prop:Ito} as in Theorem 8.3.4 of \cite{peng_nonlinearExpectation_book}, where an It\^o formula is extended from functions with bounded derivative to $C^{2}$-functions by using stopping time arguments and passing to the limit. The problem is that stopping times are very delicate in the $G$-setting, and the available results are not applicable in our setting as the $B^H$ is not a $G$-It\^{o} process. }
\end{remark}}

\blue{
We now want to apply the definition of the pathwise integral \eqref{eq:DefinitionIntegralfGBMPathwise} and the associated It\^{o} formula \eqref{eq:Ito1} to discuss the existence of pathwise arbitrage opportunities in a financial market driven by a fractional $G$-Brownian motion of Hurst index  $H \in (\frac{1}{2}, 1)$. }

\blue{In the classical case \greenNew{with} no volatility uncertainty, this question has been studied in several works, see e.g. \cite{rogers_arbitrage_1997}, \cite {dasgupta} and \cite{shiryaev1998arbitrage}: the existence of arbitrage strategies, according to the well-known fundamental theorem of asset pricing in \cite{Delbaen_schachermayer}, results from the fact that \blue{a} fractional Brownian motion of Hurst index $H\in (\frac{1}{2}, 1)$ is not a semimartingale. Later on, the so-called Wick-It\^o integral with respect to a fractional Brownian motion for $H \in  (\frac{1}{2}, 1)$ has been introduced in \cite{Hu_Oksendal_2003} and extended for $H \in (0,1)$ in \cite{Elliot_vanDerHoek}. Based on this calculus, it can be shown that there does not exist any Wick-self-financing arbitrage in financial markets driven by a fractional Brownian motion. However, the definition of a Wick-self-financing strategy lacks a convincing economic interpretation. For alternative considerations about arbitrage in fractal models, we refer to \cite{bayraktar2005arbitrage}, \cite{bender_sottinen_valkeila}, \greenNew{\cite{cheridito_2003}}, \cite{blackledge2021areview}, \cite{deng2024statistical}, \cite{djeutcha2019solving},  \cite{Guasoni_2005}. }

\blue{In the $G$-expectation setting, the definition of arbitrage has been introduced in \cite{vorbrink}. The next example combines the definition of an arbitrage under volatility uncertainty with the one of a pathwise arbitrage and proves the existence of arbitrage strategies in fractal models under volatility uncertainty for $H \in  (\frac{1}{2}, 1)$.  This can be seen as an extension under volatility uncertainty of the example provided in Section 2 of \cite{shiryaev1998arbitrage}, where the existence of arbitrage opportunities is shown for a fractal Bachelier model.
\begin{example}
	 Consider a financial market consisting of a risk-free asset $S^0=(S^0_t)_{t \in [0,T]}$ defined by $S_t^0:=e^{rt}$ for $r \geq 0$ and of a risky asset $S^1=(S^1_t)_{t \in [0,T]}$ with $S_t^1:=S_0+ \mu t + B_t^H$, where $\mu \in \mathbb{R}$ and $B^H$ is a one-dimensional fractional $G$-Brownian motion of Hurst index $H \in (\frac{1}{2},1)$. \\
	 Define the wealth process $V^{\xi}=(V^{\xi}_t)_{t \in [0,T]}$ associated to an $\mathbb{F}$-adapted process $\xi:=(\xi^0, \xi^1)=(\xi_t^0, \xi_t^1)_{t \in [0,T]}$, intepreted as the trading strategy, by
	\begin{align*}
	V^{\xi}_t:= \xi^0_t S_t^0 + \xi_t^1 S^1_t, \quad t \in [0,T].
	\end{align*}
	In particular, $\xi$ is called \emph{pathwise admissible} if it is \emph{pathwise self-financing}, i.e. 
	\begin{align}
		V^{\xi}_t&= V^{\xi}_0+ \int_0^t  \xi^0_u d  S_u^0 +  {\int_0^t\xi_u^1 dS^1_u} \nonumber \\
		&=V_0^{\xi} + \int_0^t \xi_u^0 S_u^0 r du + \int_0^t \xi_u^1 (\mu du + dB_u^H) \label{eq:IntegralSelfFinancing},
	\end{align} 
	where the integral in \eqref{eq:IntegralSelfFinancing} is defined in the sense of Definition \ref{def:IntegralFractionalGBrownianMotion} for suitable processes $\xi$, and if $V^{\xi}$ is lower bounded.
	By Remark \ref{remark:ExplicitConstruction}, it is possible to construct a set of probability measures $\mathcal{P}$ associated to the uncertainty set $\Theta$ in  \eqref{G-equationOneDimension}. A pathwise admissible trading strategy is an \emph{arbitrage} with respect to  $\mathcal{P}$ if the associated wealth process $V^{\xi}$ satisfies
	\begin{enumerate}
		\item $ V^{\xi}_0=0$\blue{;}
		\item $ V^{\xi}_T \geq 0$ quasi-surely\blue{;}
		\item there exists $P \in \mathcal{P}$ in \eqref{eq:GExpectationUpperExpectation} such that $P(V_T^{\xi}>0)>0$. 
	\end{enumerate}
	This definition is in line with the definition of an arbitrage under volatility uncertainty introduced in \cite{vorbrink}.
	From now on, we assume by simplicity that $r=0$, $S_0=1$ and $\mu=0$. We fix two constants $\underline{\sigma}$, $\overline{\sigma}$ with $0<\underline{\sigma}\le1\le\overline{\sigma}$ and let $\Theta:=[\underline{\sigma}^2, \overline{\sigma}^2]$.
 We then claim that the trading strategy
	\begin{equation} \label{eq:TradingStrategy}
	\xi_t=\left(-(B_t^H)^2-2 B_t^H,2B_t^H\right), \quad t \in [0,T],
	\end{equation}
	is an arbitrage strategy with respect to the set $\mathcal{P}$ associated to $\Theta$. First note that 
	\begin{align*}
		V_t^{\xi}=-(B_t^H)^2-2 B_t^H+ 2 B_t^H (1+ B_t^H)= (B_t^H)^2, \quad 0 \le t \le T,
	\end{align*}
	which is bounded from below by $0$. Additionally, by applying Proposition \ref{prop:Ito} for $f(x)=x^2$ we get 
	$$
	V_t^\xi=\int_0^t 2 B_u^Hd B_u^H= V_0^{\xi} +\int_0^t \xi_u^1 dB_u^H,
	$$
	i.e. $\xi$ is pathwise self-financing and therefore \emph{pathwise admissible}. Obviously, $V_0^{\xi}=0$ and $V_T^{\xi}\geq 0$ quasi-surely.  Moreover, again according to Remark \ref{remark:ExplicitConstruction}, the set $\mathcal{P}$ includes \greenNew{the law $P^1$ of a standard $d$-dimensional Brownian motion $W$.} By \cite{shiryaev1998arbitrage} it follows that $P^1(V_T^{\xi}>0)>0.$  Therefore, \greenNew{$\xi$ as in \eqref{eq:TradingStrategy}} is an arbitrage opportunity with respect to $\mathcal{P}$.
	\end{example}
}

%\bibliography{2024_11_30_fBM_G-Setting.bib}{}
%\bibliographystyle{plain}

\end{document}